\documentclass[reqno,11pt]{amsart}
\usepackage[utf8]{inputenc}
\usepackage{amsmath, latexsym, amsfonts, amssymb, amsthm, amscd}
\usepackage{graphics,epsf,psfrag}
\usepackage{graphicx}
\usepackage[dvipsnames]{xcolor}
\usepackage{url}

\numberwithin{equation}{section}

\newtheorem{theorem}{Theorem}[section]
\newtheorem{lemma}[theorem]{Lemma}

\newtheorem{rem}[theorem]{Remark}
\newtheorem{definition}[theorem]{Definition}

\renewcommand{\tilde}{\widetilde}
\xdefinecolor{red}{named}{Red}
\xdefinecolor{blue}{named}{blue}

\newcommand{\cF}{{\ensuremath{\mathcal F}} }
\newcommand{\cP}{{\ensuremath{\mathcal P}} }

\newcommand{\cC}{{\ensuremath{\mathcal C}} }

\newcommand{\cG}{{\ensuremath{\mathcal G}} }

\newcommand{\bE}{{\ensuremath{\mathbf E}} }

\newcommand{\dd}{\,\text{\rm d}}             

\newcommand{\bbE}{{\ensuremath{\mathbb E}} }

\newcommand{\bbN}{{\ensuremath{\mathbb N}} }

\newcommand{\bbR}{{\ensuremath{\mathbb R}} }
\newcommand{\bbS}{{\ensuremath{\mathbb S}} }

\newcommand{\ga}{\alpha}

\newcommand{\gd}{\delta}
\newcommand{\gep}{\varepsilon}       

\newcommand{\gl}{\lambda}

\newcommand{\gs}{\sigma}

\makeatletter
\def\captionfont@{\footnotesize}
\def\captionheadfont@{\scshape}

\long\def\@makecaption#1#2{%
  \vspace{2mm}
  \setbox\@tempboxa\vbox{\color@setgroup
    \advance\hsize-6pc\noindent
    \captionfont@\captionheadfont@#1\@xp\@ifnotempty\@xp
        {\@cdr#2\@nil}{.\captionfont@\upshape\enspace#2}%
    \unskip\kern-6pc\par
    \global\setbox\@ne\lastbox\color@endgroup}%
  \ifhbox\@ne 
    \setbox\@ne\hbox{\unhbox\@ne\unskip\unskip\unpenalty\unkern}%
  \fi
  \ifdim\wd\@tempboxa=\z@ 
    \setbox\@ne\hbox to\columnwidth{\hss\kern-6pc\box\@ne\hss}%
  \else 
    \setbox\@ne\vbox{\unvbox\@tempboxa\parskip\z@skip
        \noindent\unhbox\@ne\advance\hsize-6pc\par}%
\fi
  \ifnum\@tempcnta<64 
    \addvspace\abovecaptionskip
    \moveright 3pc\box\@ne
  \else 
    \moveright 3pc\box\@ne
    \nobreak
    \vskip\belowcaptionskip
  \fi
\relax
}
\makeatother
\def\writefig#1 #2 #3 {\rlap{\kern #1 truecm
\raise #2 truecm \hbox{#3}}}

\newcommand{\dist}{\text{\rm dist}}

\newcommand{\Var}{\mathbf{Var}}
\newcommand{\Cov}{\mathbf{Cov}}
\newcommand{\proj}{\text{\rm proj}}

\title[Periodicity in the kinetic mean-field FitzHugh-Nagumo model]
{Periodicity induced by noise and interaction in the kinetic mean-field FitzHugh-Nagumo model}

\author{Eric Lu\c{c}on}
\address{Laboratoire MAP5 (UMR CNRS 8145), Universit\'e Paris Descartes, Sorbonne Paris Cit\'e, 75270 Paris, France, \url{eric.lucon@paridescartes.fr}.
}

\author{Christophe Poquet}
\address{Univ Lyon, Université Claude Bernard Lyon 1, CNRS UMR 5208, Institut Camille Jordan, F-69622 Villeurbanne, France, \url{poquet@math.univ-lyon1.fr}}


\date{\today}

\begin{document}

\begin{abstract}
We consider the long-time behavior of a population of mean-field oscillators modeling the activity of interacting excitable neurons in large population. Each neuron is represented by its voltage and recovery variables, which are solution to a FitzHugh-Nagumo system, and interacts with the rest of the population through a mean-field linear coupling, in the presence of noise. The aim of the paper is to study the emergence of collective oscillatory behaviors induced by noise and interaction on such a system. The main difficulty of the present analysis is that we consider the kinetic case, where interaction and noise are only imposed on the voltage variable. We prove the existence of a stable cycle for the infinite population system, in a regime where the local dynamics is small. 
\end{abstract}

\maketitle

\section{Introduction}

\subsection{A mean-field model of interacting FitzHugh-Nagumo neurons}
We are interested in this paper in the behavior as $t\to\infty$ of the following McKean-Vlasov process
\begin{equation}\label{eq:kinetic FhN}
\left\{\begin{array}{l}
\dd X_t = \gd \left( X_t-\frac{X^3_t}{3} -Y_t\right)\dd t - K\left(X_t-\bE[X_t]\right)\dd t+\sqrt{2} \gs \dd B_t\\
\dd Y_t = \frac{\gd}{c}\left(X_t+a-bY_t\right)\dd t
\end{array}
\right. ,\ t\geq0,
\end{equation}
where $(X_{ t}, Y_{ t})\in \mathbb{ R}^{ 2}$, $B_{ t}$ is a standard Brownian motion on $ \mathbb{ R}$, $a\in \mathbb{ R}$ and $b,c, \delta, K, \sigma$ are positive parameters. The evolution \eqref{eq:kinetic FhN} is a prototype of a nonlinear stochastic differential equation (the nonlinearity coming from the fact that $X_{ t}$ interacts with its own law through its expectation $ \mathbf{ E} \left[X_{ t}\right]$). It is named kinetic by 
analogy to the classical kinetic interacting particle systems, noise and interactions are only applied on the "momentum" $X_t$, and not on the "position" $Y_t$. 

The system \eqref{eq:kinetic FhN} is the natural macroscopic limit (as $n\to\infty$) of the following system of coupled mean-field diffusions $(X_{ i, t}, Y_{ i, t})$, $i=1, \ldots, n$, $n\geq1$
\begin{equation}
\label{eq:particle_syst}
\left\{\begin{array}{l}
\dd X_{ i, t} = \gd \left( X_{ i, t}-\frac{X^3_{ i, t}}{3} -Y_{ i, t}\right)\dd t - K\left(X_{ i, t}- \frac{ 1}{ n} \sum_{ j=1}^{ n}X_{ j, t}\right)\dd t+\sqrt{2} \gs \dd B_{ i,t}\\
\dd Y_{ i, t} = \frac{\gd}{c}\left(X_{ i, t}+a-bY_{ i, t}\right)\dd t
\end{array}
\right. ,\ t\geq0\ ,
\end{equation}
where $(B_{ 1}, \ldots, B_{ n})$ are i.i.d. standard Brownian motions. The motivation comes from neuroscience: \eqref{eq:particle_syst} models the evolution of $n$ neurons of FitzHugh-Nagumo type, each represented by its voltage $X_{ i}$ and recovery variable $Y_{ i}$, that are coupled through a linear mean-field interaction (this corresponds to a coupling via electrical synapses, see \cite{MR3392551}). 
\subsection{Emergence of collective structured dynamics for excitable systems}
In \eqref{eq:particle_syst}, the intrinsic dynamics of each neuron is of FitzHugh-Nagumo type  \cite{FitzHugh1961,MR1779040,22657695}: when $K= \sigma=0$, $ \delta=1$, the system \eqref{eq:particle_syst} reduces to a collection of copies of the isolated system
\begin{equation}\label{eq:isolated FhN}
{\rm d}(X_{ t}, Y_{ t})= F(X_{ t}, Y_{ t}) {\rm d}t\ , 
\end{equation}
where $F$ is given by 
\begin{equation}
\label{eq:FHN}
F(x, y):= \begin{pmatrix}
f(x, y) \\ g(x, y)
\end{pmatrix}:= \begin{pmatrix}
x-\frac{x^3}{3} - y\\ \frac{1}{c}\left(x+a-by\right)
\end{pmatrix},\ x, y\in \mathbb{ R}.
\end{equation}
Although the transitions in the FitzHugh-Nagumo model are complex in general (see \cite{MR1779040}), two main dynamical patterns emerge for \eqref{eq:isolated FhN}: a resting state (corresponding to a unique stable point for \eqref{eq:isolated FhN}) and a spiking regime (corresponding to a limit cycle for \eqref{eq:isolated FhN}, see \cite{MR1779040} and Figure~\ref{fig:F} below). In this sense, the system \eqref{eq:isolated FhN} is a prototype of an excitable dynamics \cite{LINDNER2004321}: it is possible to choose appropriately the parameters $a,b,c$ (and we will do so in the following) so that the unperturbed system \eqref{eq:isolated FhN} is in a resting state but such that the addition of a small perturbation makes the system fall into a oscillatory regime (spiking activity).
\begin{figure}[h]
\centering
\includegraphics[width=\textwidth]{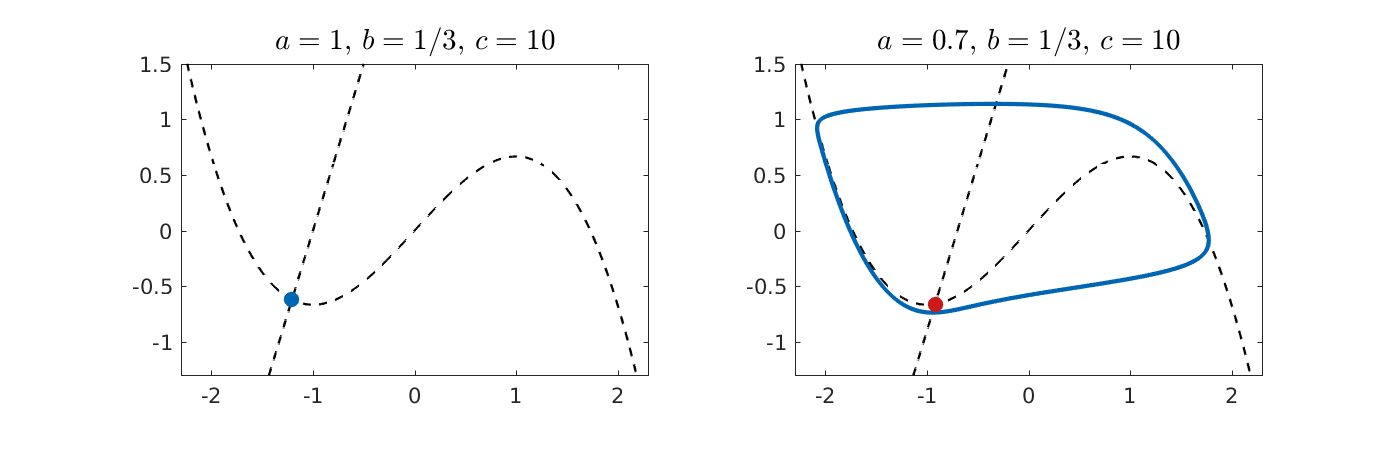}
\caption{Phase diagrams for the system \eqref{eq:isolated FhN} for two choices of parameters $a$, $b$ and $c$ (the voltage $X$ is represented along the $x$-axis and the recovery variable $Y$ on the $y$-axis). Stable (resp. unstable) points and limit cycles are represented in blue (resp. red). The nullclines of the FitzHugh-Nagumo system \eqref{eq:isolated FhN} are represented in dashed lines.}
\label{fig:F}
\end{figure}

Our aim is to analyse the joint influence of noise and interaction on the emergence of collective periodic behaviors for such a system of coupled excitable units. The point of the paper is to address this issue at the level of \eqref{eq:kinetic FhN}, that is for an infinite population $n=\infty$.  Note that this problem can be equivalently considered at the level of the Fokker-Planck PDE associated to \eqref{eq:kinetic FhN}: the law $ \mu_{ t}= \mathcal{ L}(X_{ t}, Y_{ t})$ of the McKean-Vlasov process \eqref{eq:kinetic FhN} is a weak solution to the following nonlinear kinetic Fokker-Planck PDE
\begin{multline}\label{eq:PDER kinetic FhN}
\partial_t \mu_t=\gs^2 \partial^2_{x^2}\mu_t-\partial_x\left(\left[\gd\left(x-\frac{x^3}{3}-y\right)-K\left(x-\int_{\bbR^2}z_{ 1}\mu_t( {\rm d}z_{ 1}, {\rm d}z_{ 2}) \right) \right]\mu_t \right)\\-\frac{\gd}{c} \partial_y\left((x+a-by)\mu_t\right),\ t\geq0\ ,
\end{multline}
whose solution $ t \mapsto \mu_{ t}({\rm d}x, {\rm d}y)$ takes its values in the set of probability measures on $ \mathbb{ R}^{ 2}$. Equivalently the unique solution $ \mu$ to \eqref{eq:PDER kinetic FhN} such that $ \mu_{ \vert_{ t=0}}= \mu_{ 0}$ is the law of the process \eqref{eq:kinetic FhN} with $ \mu_{ 0}= \mathcal{ L}(X_{ 0})$. Well-posedness results concerning both \eqref{eq:PDER kinetic FhN} and \eqref{eq:kinetic FhN} (in appropriate $L^{ 2}$-spaces with exponential weights) are addressed in \cite{Mischler2016,Lucon:2018qy}.
\begin{rem}
\label{rem:parameters}
Observe here that we have in \eqref{eq:kinetic FhN}, \eqref{eq:particle_syst} and \eqref{eq:PDER kinetic FhN} an interplay between three parameters: the strength of interaction $K>0$, the intensity of noise $ \sigma>0$ and the scaling parameter $ \delta>0$ of the local dynamics. In fact, a simple time change in \eqref{eq:kinetic FhN} shows that only two of these parameters are really relevant: we analyse below the long-time dynamics of \eqref{eq:kinetic FhN} in terms of $ \delta$ and $ \frac{ \sigma^{ 2}}{ K}$.
\end{rem}

Our aim is to analyse the emergence of synchronicity in the mean-field system \eqref{eq:kinetic FhN} (or equivalently \eqref{eq:PDER kinetic FhN}) under the joint influence of noise and interaction: the main result of the paper (see Theorem~\ref{th:main} below) concerns the existence of a stable invariant cycle for \eqref{eq:PDER kinetic FhN} in a regime where the interaction and noise are nontrivial and large w.r.t. to the intrinsic dynamics :
\begin{equation}
\label{eq:regime_param}
0<\frac{ \sigma^{ 2}}{K}<\infty \text{ and } \delta\ll 1.
\end{equation}
The first condition of \eqref{eq:regime_param} is a natural regime for the emergence of collective oscillations for \eqref{eq:kinetic FhN}: informally, when $K=0$ (that is in absence of interaction) for every isolated units \eqref{eq:isolated FhN} in an excitable state (first frame of Figure~\ref{fig:F}), the addition of noise alone make them leave the resting state (leading to a variety of uncorrelated stochastic dynamical patterns, e.g. canard-type excursion, mixed-mode oscillations, etc., see \cite{LINDNER2004321,MR2197663} for further references). It is only when one adds further some nontrivial interaction $K>0$ that these excursions may happen collectively, leading in the $n\to\infty$ limit to global oscillations of the system. Note that our main result is sufficiently versatile to track carefully how much noise (for a given $K$) one has to put in the system in order to see oscillations: we refer to Section~\ref{sec:emergence_intro} where we describe bifurcations of \eqref{eq:kinetic FhN} in terms of $ \sigma$ (from $ \sigma=0$ to $ \sigma\to\infty$).

Different asymptotics have been considered in previous works for the same model. We refer to Section~\ref{sec:comments} below for more details.

\subsection{Slow-fast dynamics approach}
\label{sec:slow-fast approach}
The second hypothesis $\delta\ll 1$ in \eqref{eq:regime_param} comes from the fact that our approach relies on a perturbative slow-fast analysis. We prove namely that when $\gd \ll 1$, the system \eqref{eq:kinetic FhN} admits solutions that are close to having Gaussian distributions. We present here the heuristic of this reduction, which will be made rigorous in the next sections.

For the variable $X_t$ in \eqref{eq:kinetic FhN}, when $\gd$ is small (and $K$ and $\gs$ remain of order $1$), the interaction and noise terms constitute a fast part of the dynamics, while the intrinsic dynamics term $ \delta F$ constitutes a slow one. When $ \delta=0$, the fast dynamics part of \eqref{eq:kinetic FhN} simply reduces an Ornstein Ulhenbeck process (of constant expectation) with Gaussian invariant measure. Hence, when $ \delta$ is now positive but small, it is natural to approximate at first order in $ \delta$, the distribution of $X_t$ by a Gaussian distribution $\mathcal{N}( \mathbf{ E}[X_t],\gs^2/K)$, where $\bE[X_t]$ evolves slowly in time.

Now if $X_t$ is at first order Gaussian, so is $Y_t$ since its dynamics is linear. So, at first order in $ \delta$, $(X_t,Y_t)$ should have a Gaussian distribution $\mathcal{N}(m_t,\Gamma_\gd)$, where $\Gamma_\gd$ is a symmetric covariance matrix, and $m_t=(\bE[X_t],\bE[Y_t])$ satisfies
\begin{equation}
\dot m_t = \gd\left(
\begin{array}{c}
\bE[X_t]-\frac{\bE[X_t^3]}{3}-\bE[Y_t]\\
\frac{1}{c}\left( \bE[X_t]+a-b\bE[Y_t]\right)
\end{array}
\right).
\end{equation}
Considering that $X_t\approx \mathcal{ N}(x_{ t}, \sigma^{ 2}/K)$, we obtain $\bE[X_t^3]\approx \bE[X_t]^3 +3\frac{\gs^2}{K}$, which leads to the approximation
\begin{equation}\label{eq:approx dym mt}
\dot m_t \approx \gd F_{\frac{\gs^2}{K}}(m_t),
\end{equation}
with
\begin{equation}
\label{eq:FHN_u}
F_{ u}(x, y):= \begin{pmatrix}
f_{ u}(x, y) \\ g(x, y)
\end{pmatrix}:= \begin{pmatrix}
(1-u)x-\frac{x^3}{3} - y\\ \frac{1}{c}\left(x+a-by\right)
\end{pmatrix},\ x, y\in \mathbb{ R}.
\end{equation}
Note that \eqref{eq:FHN_u} is once again of FitzHugh-Nagumo type, only modified by the prefactor $(1-u)$ in front of the $x$ variable. 

To compute $\Gamma_\gd$, we denote $Z_t=(Z^x_t,Z^y_t)$ a first order approximation of the centered process $(X_t-\bE[X_t], Y_t-\bE[Y_t])$, defined by the following system of equations:
\begin{equation}
\label{eq: def Z}
\left\{\begin{array}{l}
\dd Z^x_t = - K Z^x_t \dd t+\sqrt{2} \gs \dd B_t\\
\dd Z^y_t = \frac{\gd}{c}\left(Z^x_t-b Z^y_t\right)\dd t
\end{array}
\right. .
\end{equation}
Remark that due to the fact that this dynamics is linear, if the distribution of $Z_0$ is a Gaussian, then the distribution of $Z_t$ remains Gaussian, and straightforward calculations lead to
\begin{align}
\frac{\dd}{\dd t} \Var[Z^y_t]&=\frac{2\gd}{c}\Cov[Z^x_t,Z^y_t] -\frac{2 b \gd}{c}\Var[Z^y_t],\\
\frac{\dd}{\dd t} \Cov[Z^x_t,Z^y_t]&=\frac{\gd}{c} \Var[Z^x_t]-\left(K+\frac{b\gd}{c}\right)\Cov[Z^x_t,Z^y_t].
\end{align}
The equilibrium solution of this system of equations is given by $\Var[Z^x_t]=\frac{\gs^2}{K}$, $\Var[Z^y_t]=\frac{\gs^2}{K}\frac{\gd}{b(Kc+b\gd)}$ and $\Cov[Z^x_t,Z^y_t]=\frac{\gs^2}{K}\frac{\gd}{Kc+b\gd}$, and we thus define
\begin{equation}
\Gamma_\gd= \frac{\gs^2}{K}\left(\begin{array}{cc}
1&\frac{\gd}{Kc+b\gd}\\\frac{\gd}{Kc+b\gd} &\frac{\gd}{b(Kc+b\gd)}
\end{array}\right).
\end{equation}

\begin{rem} A similar slow-fast analysis was made in \cite{Lucon:2018qy} for the elliptic case (that is with nontrivial interaction $K_{ 2}>0$ and noise $ \sigma_{ 2}>0$ on the $Y$-variable)
\begin{equation}\label{eq:elliptic FhN}
\left\{\begin{array}{l}
\dd X_t= \delta f(X_{ t}, Y_{ t})\dd t - K_{ 1}\left(X_t-\bE[X_t]\right)\dd t+\sqrt{2} \gs_{ 1} \dd B_{ 1, t}\\
\dd Y_t= \delta g(X_{ t}, Y_{ t})\dd t- K_{ 2}\left(Y_t-\bE[Y_t]\right)\dd t+\sqrt{2} \gs_{ 2} \dd B_{ 2, t}
\end{array}
\right. ,
\end{equation}
where $K_{ 1}, K_{ 2}, \sigma_{ 1}, \sigma_{ 2}>0$ and $B_{1,t}$ $B_{2,t}$ independent. For this elliptic case both $X_t$ and $Y_t$ have fast terms in their dynamics, and thus the distribution of $(X_t,Y_t)$ is at first order a Gaussian $\mathcal{N}(m_t, \Gamma)$, with $\Gamma=\left(\begin{array}{cc}\frac{\gs_1^2}{K_1} & 0 \\ 0 & \frac{\gs_2^2}{K_2}\end{array} \right)$ and $\dot m_t \approx F_{\frac{\gs_1^2}{K_1}}(m_t)$, with the same function $F_u$, defined by \eqref{eq:FHN_u}. In particular, the fact that $F_{\frac{\gs_1^2}{K_1}}$ does not depend on $(K_{ 2}, \sigma_{ 2})$ in the elliptic case is again a strong argument in favor of the validity of the approximation \eqref{eq:approx dym mt} in the kinetic case. However, it is not possible to make $K_{ 2}, \sigma_{ 2}\to 0$ in the arguments of \cite{Lucon:2018qy}, as they rely strongly on the non-degeneracy of the noise on both variables. Secondly, contrary to \cite{Lucon:2018qy}, where the slow-fast reduction is based on the geometric properties of the PDE \eqref{eq:PDER kinetic FhN} (in a suitable $L^{ 2}$-space with exponential weights), we focus mostly here on the properties of the system \eqref{eq:kinetic FhN}, relying on Wasserstein type estimates.
\end{rem}

\subsection{Collective oscillations in the FitzHugh-Nagumo model under noise and interaction}
\label{sec:emergence_intro}
Supposing that the analysis made above is rigorous, i.e. that the coupled mean-field system \eqref{eq:kinetic FhN} admits solutions $(X_t,Y_t)$ that are at first order Gaussian distribution $\mathcal{N}(m_t,\Gamma_\gd)$, then (at least at first order) the analysis of  \eqref{eq:kinetic FhN} can be reduced to the analysis the dynamics of its expectation $m_t$, for which we provided the first order approximation \eqref{eq:approx dym mt}. Since \eqref{eq:approx dym mt} is nothing else than another FitzHugh-Nagumo system (slowed-down by a factor $ \delta$), everything reads now in terms of the bifurcations of the system (recall \eqref{eq:FHN_u})
\begin{equation}
\label{eq:Fu}
\left(\begin{array}{c}\dot x_{t}\\ \dot y_{ t}\end{array} \right) = F_{ u}(x_{ t}, y_{ t})
\end{equation} as $u$ increases from $u=0$ (isolated system \eqref{eq:isolated FhN}) to $ u= \frac{ \sigma^{ 2}}{ K}>0$ (coupled system with noise and interaction). 
In particular, as we will see below, a crucial observation is that it is possible to choose carefully the parameters $a, b, c$ so that \eqref{eq:Fu} has a unique  stable state for $u=0$, whereas the same system exhibits oscillations for values of $u$ chosen in a bounded interval (see Figure~\ref{fig:Fu} below): this is the signature of the emergence of periodic dynamics due to noise and interaction in \eqref{eq:kinetic FhN}.

The analysis of the bifurcations of the system \eqref{eq:Fu} was already made in \cite{Lucon:2018qy}, Section~3, since it is also the dynamics obtained by slow-fast reduction of the elliptic case. We present this analysis here, for the sake of completeness. With no loss of generality (Remark~\ref{rem:parameters}), we can suppose $K=1$. It is then possible to read from Figure~\ref{fig:Fu} the behavior of the system \eqref{eq:kinetic FhN} as the noise intensity $ u= \sigma^{ 2}$ increases: in absence of noise ($u=0$), each neuron has a fixed-point dynamics. As $u$ increases, we observe the emergence (through a saddle-node bifurcation of cycles) of a stable cycle coexisting with an unstable cycle and stable fixed-point (see the case $u=0.086$) for \eqref{eq:Fu}. As $u$ increases further, the unstable cycle collides with the stable fixed-point, resulting in only one stable limit cycle coexisting with an unstable fixed-point (see the case $u=0.2$). For large noise, oscillations disappear (see $u=0.8$). We refer to \cite{Lucon:2018qy}, Section~3.4 for more details on these transitions, in terms of Hopf bifurcations, and also for a study of the bistable case.

\begin{figure}[h]
\centering
\includegraphics[width=\textwidth]{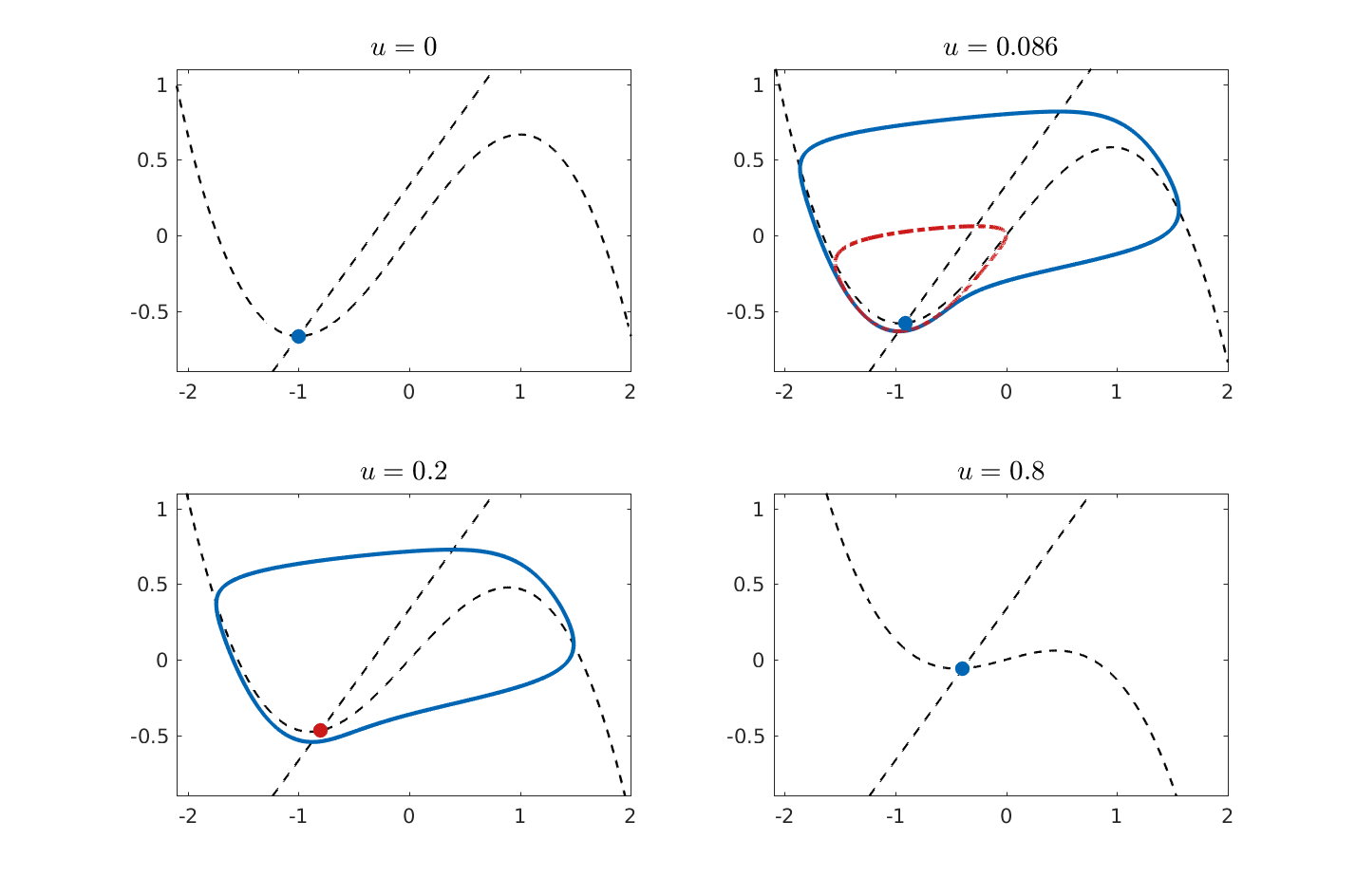}
\caption{Phase diagrams for the system \eqref{eq:Fu} for parameters $a= \frac{ 1}{ 3}$, $b=1$, $c= 10$ for different choices of $u$. Stable (resp. unstable) points and limit cycles are represented in blue (resp. red). The nullclines of the FitzHugh-Nagumo system \eqref{eq:Fu} are represented in dashed lines.}
\label{fig:Fu}
\end{figure}

As an illustration of the accuracy of this analysis for the particle system \eqref{eq:particle_syst}, we reproduce in Figure~\ref{fig:particle system} the dynamics of the empirical measure of \eqref{eq:particle_syst}, for $\gd=0.2$, $a= \frac{ 1}{ 3}$, $b=1$, $c= 10$, $K=1$, $ \sigma^2=0.2$ (which corresponds to $u=0.2$ in Figure~\ref{fig:Fu}) and $n=10^{ 5}$: the mean-value of the empirical density follows precisely the limit cycle given by \eqref{eq:Fu} in the case $u=0.2$. One notable difference between the present simulation in the kinetic case and the simulations in \cite{Lucon:2018qy} in the elliptic case (see \cite{Lucon:2018qy}, Figure~7) is the shape of the empirical density in both cases: due to the absence of noise on the $y$-coordinate, the variance along the $y$-direction is here significantly smaller than in the elliptic case.

\begin{figure}[h]
\centering
\includegraphics[width=\textwidth]{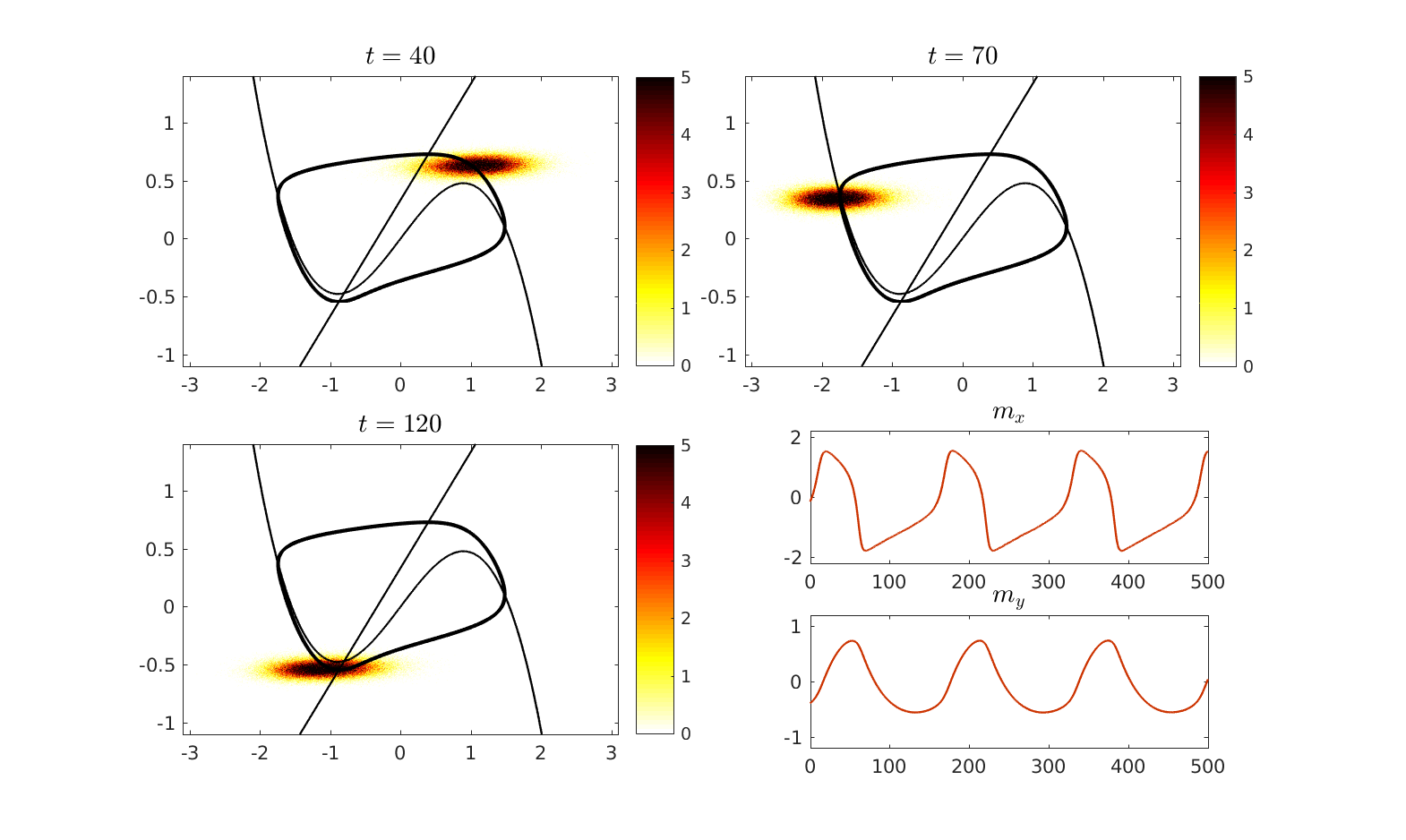}
\caption{Time evolution of the empirical density of the particle system \eqref{eq:particle_syst}, in the case $\gd=0.2$, $a= \frac{ 1}{ 3}$, $b=1$, $c= 10$, $K=1$, $ \sigma^2=0.2$ and $n=100000$. The corresponding evolution of the empirical mean-value is represented.}
\label{fig:particle system}
\end{figure}

\subsection{Main results}
In the remaining of the paper, we fix the parameters $(a,b,c, \sigma, K)$ such that the system
\begin{equation}\label{eq:FhN gs2/K}
\frac{1}{\gd} \left(
\begin{array}{c}
\dot x_t\\
\dot y_t
\end{array}
\right)= F_{\frac{\gs^2}{K}}(x_t,y_t),
\end{equation}
admits a limit cycle (recall that this is possible even if it is not the case for the system \eqref{eq:isolated FhN}, see Section~\ref{sec:emergence_intro} and \cite{Lucon:2018qy}, Section~3.4). 
\begin{rem}
\label{rem:gamma_delta}
Note that the existence of such limit cycle does not depend on $ \delta$: if, in the case $ \delta=1$, we denote this stable periodic solution by $(\gamma_{t})_{t\in [0,T_\gamma)}$ ($T_\gamma$ being the period of the limit cycle for $\gd=1$), the behavior of \eqref{eq:FhN gs2/K} for general $ \delta>0$ may be deduced by a simple time change: the corresponding periodic orbit becomes $ \left( \gamma_{t}^{ \delta}\right):=(\gamma_{\gd t})_{t\in [0,T_\gamma/\gd )}$, with period $ \frac{ T_{  \gamma}}{ \delta}$. When no confusion is possible, we will drop the superscript $ \delta$ in $ \gamma^{ \delta}$ in the following. We adopt also the definition
\begin{equation}
\mathbb{ S}_{ \delta}:= \mathbb{ R}/\big((T_{ \gamma}/ \delta) \mathbb{ Z}\big)
\end{equation}
endowed with the quotient topology induced by the euclidean distance on $ \mathbb{ R}$. Such topology can be generated by the following metric: for all $ \bar \varphi \equiv \varphi \left[ \frac{ T_{ \gamma}}{ \delta}\right]$, $ \bar \psi \equiv \psi \left[ \frac{ T_{ \gamma}}{ \delta}\right]$, 
\begin{equation}
d_{ \mathbb{ S}}( \bar \varphi, \bar \psi):= \max \left( \left\vert \varphi - \psi \right\vert, \frac{ T_{ \gamma}}{ \delta} - \left\vert \varphi - \psi\right\vert\right)\, .
\end{equation}
\end{rem}

We denote $q_m$  the distribution given by
\begin{equation}
q_m(z)= q_{ m}^{ \delta}(z):=\frac{1}{2\pi \det(\Gamma_{ \delta})}\exp\left(\frac12 (z-m)\cdot \Gamma_\gd^{-1} (z-m)\right), \quad z\in \bbR^2,
\end{equation}
and
\begin{equation}
\label{eq:G_delta}
\cG_\gd:=\left\{q^\gd_{\gamma^\gd_t}:\, t\geq 0\right\}.
\end{equation}
Below, $W_{ 2}$ is a distance of Wasserstein-type that is precisely defined in Definition~\ref{def:wasserstein}. The main result of the paper is the following:
\begin{theorem}\label{th:main}
Choose parameters $a$, $b$, $c$, $K$ and $\gs^2$ such that \eqref{eq:FhN gs2/K} admits an stable limit cycle. Then there exists a $\gd_c>0$ such that for all $\gd\leq \gd_c$ there exists a periodic solution $\nu^\text{per}_t$ to \eqref{eq:PDER kinetic FhN}, defining an invariant cycle $\cC_\gd$ which satisfies $\sup_t\dist_{W_2}(\nu^{\text{per}}_t,\cG_\gd)=O(\gd)$.

Moreover there exist positive constants $C_1$ and $C_2$ that do not depend on $\gd$, a positive constant $C(\gd)$ and a positive rate $\gl(\gd)$ such that if $\mu_0\in \cP_2$ satisfies
\begin{equation}
\int_{\bbR^2} z^6 \dd \mu_0(z)\leq C_1, \qquad \text{and} \qquad \dist_{W_2}(\mu_0,\cC_{\gd})\leq C_2\gd,
\end{equation}
then for $t\mapsto\mu_t$ the solution to \eqref{eq:PDER kinetic FhN} with initial condition $\mu_0$ we have:
\begin{equation}
\dist_{W_2}(\mu_t,\cC_\gd)\leq C(\gd) e^{-\gl(\gd) t} \dist_{W_2}(\mu_0,\cC_\gd).
\end{equation}
\end{theorem}

\begin{rem} The constants $\lambda(\gd)$ and $C(\gd)$ obtained in the proof of Theorem \ref{th:main} satisfy
\begin{equation}
\gl(\gd)\underset{\gd\rightarrow 0}{\longrightarrow} \min\left\{ \gl, \frac{b}{c}\right\},\qquad \text{and} \qquad C(\gd)\underset{\gd\rightarrow 0}{\longrightarrow} \infty,
\end{equation}
where $\gl$ is the exponential rate of attraction of the limit cycle of \eqref{eq:FhN gs2/K}, in the case $ \delta=1$.
\end{rem}

\begin{rem}
We focus in this work on the proof of the existence of a stable cycle for \eqref{eq:PDER kinetic FhN} when \eqref{eq:FhN gs2/K} admits one, but it is clear that one can prove the existence of a stable fixed point for \eqref{eq:PDER kinetic FhN} when \eqref{eq:FhN gs2/K} admits one by following the same arguments (in fact simpler arguments, no need of Floquet theory in that case).
\end{rem}

\subsection{Comments and existing literature}
\label{sec:comments}
\subsubsection{On the mean-field FitzHugh-Nagumo model}
The structure of the proof we provide in this paper, relying on a fixed-point theorem, is inspired from the classical proofs of the theory of persistence of normally hyperbolic invariant manifold \cite{fenichel1971persistence,fenichel1979geometric,hirsch1977invariant,Bates1998,sell2013dynamics,wiggins2013normally}, which we could not apply directly here due to the singularity of our problem (when $\gd=0$ the dynamics of $Y_t$ is trivial and thus there can not be any stable compact invariant manifold for \eqref{eq:PDER kinetic FhN}).

Several recent other works have analyzed the long-time behavior of the FitzHugh-Nagumo Fokker-Planck PDE \eqref{eq:PDER kinetic FhN}. The paper \cite{Mischler2016} considers the situation of small interaction (i.e. $ \delta>0$ and $ \sigma>0$ fixed with $K\to0$). Along with some well-posedness estimates concerning \eqref{eq:PDER kinetic FhN}, the main result of \cite{Mischler2016} concerns the existence of stationary states for \eqref{eq:PDER kinetic FhN} in the limit of small interaction. The case where $K\to\infty$ with $ \delta$ and $ \sigma$ fixed is analyzed in \cite{2018arXiv180406758Q}. The authors prove concentration in large time around singular solutions (clamping) in such a case. In a regime similar to ours, one should finally mention the recent \cite{2019arXiv190501342T} which analyzes the possible microscopic mechanisms responsible for the emergence of collective oscillations.

As already said, this paper addresses the case of an interaction modeling electrical synapses. A common framework in neuronal models concerns interactions through chemical synapses \cite{MR3392551}, which is not covered by this work. The question of the possibility of extension of the present results to this case is open. In this direction, a recent work \cite{bossy2018synchronization} addresses synchronization issues of interacting neurons with Hodgkin-Huxley dynamics with chemical synapses.

\subsubsection{On the dynamics of the particle system \eqref{eq:particle_syst}}
Obtaining a rigorous proof of the existence of periodic behaviors in infinite population limit of mean-field interacting particle models is a problem that has been studied in different situations, as chemical reaction models, rotors models, spin models, Hawkes processes, etc., see for example \cite{scheutzow1985noise,scheutzow1986periodic,giacomin2012transitions,giacomin2014coherence,giacomin2015noise,dai2013curie,collet2016,ditlevsen2017multi}, each proof relying, as in the present work, on a drastic phase reduction of the model. 

Transposing the dynamics of the nonlinear process \eqref{eq:kinetic FhN} to the dynamics of the particle system \eqref{eq:particle_syst} is a nontrivial task, that we do not address here (see nonetheless Figure~\ref{fig:particle system} for a numerical evidence of the accuracy of the present analysis w.r.t. the particle system \eqref{eq:particle_syst}). A standard way to couple i.i.d. copies $(X^{ (i)}_{ t}, Y^{ (i)}_{ t})$ ($i=1, \ldots, n$) of \eqref{eq:kinetic FhN} with the particles $(X_{ i,t}, Y_{ i, t})$ in \eqref{eq:particle_syst} is to choose identical initial condition and noise $B_{ i}$ so that the following standard propagation of chaos estimates \cite{sznitman1991topics,Mischler2016,LucSta2014,Lucon:2018qy} holds: 
\begin{equation}
\label{eq:prop_chaos}
\sup_{ i=1, \ldots, n} \mathbf{ E} \left[\sup_{ 0\leq t \leq T} \left\vert (X_{i,t}, Y_{ i, t}) - (X^{ (i)}_{ t}, Y^{ (i)}_{ t}) \right\vert^{ 2}\right] \leq \frac{ C e^{ C T}}{ n}.
\end{equation}
This propagation of chaos result may be equivalently expressed in terms of the convergence of the empirical  measure $ \mu_{ n, t}:= \frac{ 1}{ n} \sum_{ j=1}^{ n} \delta_{ (X_{ j, t}, Y_{ j, t})}$ of the particle system \eqref{eq:particle_syst} to the solution of \eqref{eq:PDER kinetic FhN}, on any time interval $[0, T]$. From \eqref{eq:prop_chaos}, we see that \eqref{eq:kinetic FhN} gives a correct approximation of \eqref{eq:particle_syst} at least up to times $T$ of order $c \ln n$, for $c>0$ small enough. The question of the relevance of this mean-field approximation for times $T\gg \ln n$ is a long-standing issue in the literature (see e.g. \cite{Bertini:2013aa,Durmus2018} and references therein). In the case of collective periodic behaviors that is of particular interest in neuroscience, one should mention in particular  \cite{Bertini:2013aa,MR3689966} (for phase oscillators), \cite{ditlevsen2017multi} (for point processes with Hawkes dynamics) and \cite{2018arXiv181105937A} (for Curie-Weiss dynamics modeling social interactions).

\subsection{Organization of the paper}
In Section~\ref{sec:closeness}, we prove some controls on the moments of \eqref{eq:kinetic FhN} as well as some estimates of proximity of solutions of \eqref{eq:PDER kinetic FhN} to the Gaussian manifold \eqref{eq:G_delta}. Section~\ref{sec:contract} gathers the main estimates (relying in particular on Floquet theory) concerning \eqref{eq:kinetic FhN}. The main result (Theorem~\ref{th:main}) is proven in Section~\ref{sec:fixed_point}. A technical lemma is postponed to the appendix.
\section{Moment estimates and proximity to Gaussian distributions}\label{sec:closeness}
Take $R_0>0$ such that the periodic solution $\gamma^{ \delta}$ of \eqref{eq:FhN gs2/K} is strictly included in the open euclidean ball $B(0,R_0)$. Note that, by Remark~\ref{rem:gamma_delta}, $R_{ 0}$ does not depend on $ \delta$. For any initial condition to \eqref{eq:kinetic FhN} such that $ \left( \mathbb{ E} \left[X_{ 0}\right], \mathbb{ E} \left[Y_{ 0}\right]\right)\in B(0, R_{ 0})$, define the exit time
\begin{equation}
T_e(R_0)=T_{ e}^{ \delta}(R_{ 0})=\inf\{t\geq 0:\, (\bE[X_t],\bE[Y_t])\notin B(0,R_0)\}. 
\end{equation}
Here, we note that $T_{ e}(R_{ 0})$ is independent of the particular coupling of $(X_{ 0}, Y_{ 0})$, provided its marginals are fixed.

\begin{lemma}\label{lem:moments X Y}
There exist positive constants $\gd_0$, $ \kappa_0^x$ and $ \kappa_0^y$ such that if $\gd\leq \gd_0$ the following is true: for any initial condition $ (X_{ 0}, Y_{ 0})$ to \eqref{eq:kinetic FhN} such that 
\begin{equation}\label{hyp:lem first bound moments}
\bE[X_0^6]\leq \kappa_0^x, \quad \text{and} \quad  \bE[Y_0^6]\leq \kappa_0^y, 
\end{equation}
then the solution $(X_t,Y_t)$ of \eqref{eq:kinetic FhN} satisfies
\begin{equation}
\sup_{t\in \left[0,T_e(R_0)\right]} \bE[X_t^6]\leq \kappa_0^x ,\quad \text{and} \quad  \sup_{t\in \left[0,T_e(R_0)\right]} \bE[Y_t^6]\leq \kappa_0^y.
\end{equation}
\end{lemma}

\begin{proof}
For $n>0$ denote the stopping time $\tau_n=\inf\{t\geq 0: X_t^2+Y_t^2=n^2\}$,  and denote $X^{(n)}_t=X_{t\wedge \tau_n}$ and $Y^{(n)}_t=Y_{t\wedge \tau_n}$. By It\^o formula we get
\begin{multline}
\frac16 (X^{(n)}_t)^6=\frac16 (X^{(n)}_0)^6+ \int_0^{t\wedge \tau_n} \Bigg[\gd (X^{(n)}_s)^5\left( X^{(n)}_s-\frac{(X^{(n)}_s)^3}{3} -Y^{(n)}_s\right)+K\bE[X_s](X^{(n)}_s)^5\\
-K (X^{(n)}_s)^6+5\gs^2(X^{(n)}_s)^2\Bigg]\dd s+\gs \int_0^{t\wedge \tau_n}(X^{(n)}_s)^5\dd B_{s}.
\end{multline}
Using the following inequalities, for positive constant $c_1,\ldots,c_4$ that we do not give explicitly to keep notations simple,
\begin{multline}
x^6-\frac{x^8}{3}\leq c_1-\frac{x^8}{4},\quad  x^5y\leq \frac{x^8}{4}+c_2 y^6,\quad  Kx^5\bE[X_t]-K x^6\leq c_3\bE[X_t]-\frac{K x^6}{3},\\
5\gs^2 x^4-\frac{2K x^6}{3}\leq c_4 -\frac{K x^6}{3},
\end{multline}
we obtain,
\begin{equation}\label{eq:derivate EXn6}
\frac14 \frac{\dd}{\dd t}\bE[(X^{(n)}_t)^6]\leq c_1\gd +c_4+c_3 \bE[X_t] +c_2 \gd \bE[(Y^{(n)}_t)^6]-\frac{K}{3}\bE[(X^{(n)}_t)^6] .
\end{equation}
On the other hand,
\begin{equation}
\frac{c}{6\gd}(Y^{(n)}_t)^6=\frac{c}{6\gd}(Y^{(n)}_0)^6+\int_0^{\tau_n}(Y^{(n)}_s)^5\left(X^{(n)}_s+a-bY^{(n)}_s\right)\dd s,
\end{equation}
and using the inequalities, for positive constants $c_6,c_7$,
\begin{equation}
xy^5\leq c_6 x^6 +\frac{b}{3}y^6,\quad \text{and}\quad ay^5-\frac{2b}{3}y^6\leq c_7 -\frac{b}{3}y^6,
\end{equation}
we get
\begin{equation}\label{eq:derivate EYn6}
\frac{c}{6\gd} \frac{\dd}{\dd t}\bE[(Y^{(n)}_t)^6]\leq c_7 +c_6\bE[(X^{(n)}_t)^6]-\frac{b}{3}\bE[(Y^{(n)}_t)^6] .
\end{equation}
Using \eqref{eq:derivate EXn6} and \eqref{eq:derivate EYn6} and the Gr\"onwall inequality we deduce that in fact $\bE[X_t^6]+\bE[Y_t^6]<\infty$ and the estimates made above are in fact valid for $n=\infty$. In particular, for $t\in [0,T_e(R_0)]$ we have
\begin{align}
\frac14 \frac{\dd}{\dd t}\bE[X_t^6]&\leq c_1\gd +c_4+c_3 R_0 +c_2 \gd \bE[Y_t^6]-\frac{K}{3}\bE[X_t^6],\label{eq:derivate E X6} \\
\frac{c}{6\gd} \frac{\dd}{\dd t}\bE[Y_t^6]&\leq c_7 +c_6\bE[X_t^6]-\frac{b}{3}\bE[Y_t^6]\label{eq:derivate E Y6}.
\end{align}
Define now $ \kappa_0^x=\frac{4}{K}(c_1+c_3R_0+c_2)$, and $ \kappa_0^y=\frac{4}{b}(c_7+c_6 \kappa_0^x)$ and
\begin{equation}
t_0=\inf\{t\geq 0:\, \bE[(X_t^{(n)}]\geq \kappa_0^x\, \text{or}\, \bE[Y_t]\geq \kappa_0^y\}. 
\end{equation}
For all $t\in [0,t_0\wedge T_e(R_0)]$ supposing that $\bE[X_0^6]\leq \kappa_0$ and $\bE[Y_0^6]\leq \kappa_0$, we get from \eqref{eq:derivate E X6} and \eqref{eq:derivate E Y6} that
\begin{align}
\bE[X_t^6]&\leq \max\left\{ \kappa_0^x,\frac{3}{K}(c_1\gd +c_4+c_3 R_0+c_2\gd \kappa_0^y), \right\},\\
\bE[Y_t^6]&\leq \max\left\{ \kappa_0^y,\frac{3}{b}(c_7+c_6 \kappa_0^x)\right\}.
\end{align}
and these maxima are equal to $ \kappa_0^x$ and $ \kappa_0^y$ respectively as soon as $\gd\leq \min\{1,1/R_0\}$. Moreover it is clear with this choice of $ \kappa_0^x$ that if $X_t$ satisfies $\bE[X_t^6]= \kappa_0^x$, then $\frac{\dd}{\dd t}\bE[X_t^6]<0$, and the same applies for $Y_t$. So $t_0\wedge T_e(R_0)=T_e(R_0)$, which concludes the proof.
\end{proof}
\begin{definition}
\label{def:wasserstein}
For a probability distribution $\nu$ on $ \mathbb{ R}^{ 2}$, $m=\int_{\bbR^2}z \dd \nu(z)$ its expectation, we denote $\tilde \nu$ the centered version of $ \nu$: $\int_{\bbR^2}\phi(z)\dd \nu(z)=\int_{\bbR^2}\phi(z+m)\dd \tilde \nu(z)$ (for all test function $\phi$).
For two distributions $ \nu_{ 1}$ and $ \nu_{ 2}$ on $ \mathbb{ R}^{ 2}$, denote by $ \mathcal{ C}( \nu_{ 1}, \nu_{ 2})$ the set of all couplings $ \pi$ of $ \nu_{ 1}$ and $ \nu_{ 2}$, that is the set of all probability measures $ \pi( {\rm d}(x_{ 1}, y_{ 1}), {\rm d}(x_{ 2}, y_{ 2}))$ on $ \mathbb{ R}^{ 2} \times \mathbb{ R}^{ 2}$ with marginals $ \nu_{ 1}( {\rm d}(x_{ 1}, y_{ 1}))$ and $ \nu_{ 2}( {\rm d}(x_{ 2}, y_{ 2}))$. If $ \pi\in \mathcal{ C}( \nu_{ 1}, \nu_{ 2})$, we denote by $ \tilde{ \pi}\in \mathcal{ C}( \tilde{ \nu}_{ 1}, \tilde{ \nu}_{ 2})$ the corresponding coupling of the centered measures: 
\[\int_{ \mathbb{ R}^{ 2}\times \mathbb{ R}^{ 2}} \phi(z_{ 1}, z_{ 2}) \pi({\rm d}z_{ 1}, {\rm d}z_{ 2})= \int_{ \mathbb{ R}^{ 2}\times \mathbb{ R}^{ 2}} \phi(z_{ 1}+m_{ 1}, z_{ 2}+m_{ 2}) \tilde{ \pi}({\rm d}z_{ 1}, {\rm d}z_{ 2})\]

For $\beta\in (0, 1)$, we consider the following Wasserstein-type distance $W=W(\gd,b,\beta)$ on $ \mathcal{ P}( \mathbb{ R}^{ 2})$, 
\begin{align}
\label{eq:def_W_1}
W(\nu_1,\nu_2):= \inf \left\lbrace \Lambda(\pi),\ \pi\in \mathcal{ C}(\nu_{ 1}, \nu_{ 2})\right\rbrace,
\end{align}
where for a fixed $ \pi\in \mathcal{ C}( \nu_{ 1}, \nu_{ 2})$, $z_{ i}=(x_{ i}, y_{ i})\in \mathbb{ R}^{ 2}$, $i=1,2$
\begin{multline}
\label{eq:def_d}
\Lambda(\pi):=\max \Bigg(\gd^\beta \left\vert \int_{ \mathbb{ R}^{ 2}\times \mathbb{ R}^{ 2}} (z_{ 1}- z_{ 2}) \pi({\rm d}z_{ 1}, {\rm d} z_{ 2})\right\vert,\\ \sqrt{\int \left\vert x_{ 1}- x_{ 2}\right\vert^{ 2} \tilde{ \pi}({\rm d}z_{ 1}, {\rm d}z_{ 2})},\ b \sqrt{\int \left\vert y_{ 1}-y_{ 2} \right\vert^{ 2} \tilde{ \pi}( {\rm d}z_{ 1}, {\rm d}z_{ 2})}\Bigg)\, .
\end{multline}
\end{definition}
The factor $\gd^\beta$ in the definition of the distance $W$ (and later the definition of the distances $W_\theta$, see section \ref{sec:fixed_point}) will be useful to obtain the contraction property, at the end of the proof of Lemma \ref{lem:X contracts}.

\begin{rem}
\begin{enumerate}
\item The definition \eqref{eq:def_d} is of course equivalent to
\begin{equation}
\Lambda(\pi)=\max \left(\gd^\beta \left\vert \mathbf{E} \left[\left(X_{ 1}, Y_{ 1}\right)\right] - \mathbf{ E} \left[\left(X_{ 2}, Y_{ 2}\right)\right]\right\vert,\ \sqrt{\bE\left[| \tilde{ X}_1- \tilde{X}_2|^2\right] },\ b \sqrt{\bE\left[| \tilde{Y}_1- \tilde{ Y}_2|^2\right]}\right)\, ,
\end{equation}
where $ \left\lbrace \left(X_{ 1}, Y_{ 1}\right), \left(X_{ 2}, Y_{ 2}\right)\right\rbrace \sim \pi$, for $ \tilde{ X}= X - \mathbf{ E} \left[X\right]$.
\item
Note that the first term in \eqref{eq:def_d} is independent of the coupling $\pi \in \mathcal{ C}( \nu_{ 1}, \nu_{ 2})$, so that we can also write
\begin{align}
\label{eq:def_W_2}
W(\nu_1,\nu_2)=\max \left\lbrace \gd^\beta\left|\int_{\bbR^2}z \nu_1({\rm d}z)-\int_{\bbR^2}z \nu_2({\rm d}z) \right|, W \left( \tilde{ \nu}_{ 1}, \tilde{ \nu}_{ 2}\right)\right\rbrace\, .
\end{align}
\item
Remark that $W$ is equivalent to the standard Wasserstein-2 distance $W_2$: for $\gd$ small enough
\begin{equation}
\min\left(1,b^{-1} \right) W\leq  W_2\leq 3\gd^{-\beta} W.
\end{equation}
\end{enumerate}
\end{rem}

Denote by $(\tilde X_t,\tilde Y_t)$ the centered version of $(X_t,Y_t)$, that satisfies
\begin{equation}
\label{eq:evol_tildeXY}
\left\{\begin{array}{l}
\dd \tilde X_t = \gd \left( X_t-\frac{ X_t^3}{3} - Y_t\right)\dd t-\dot x_t\dd t - K \tilde X_t \dd t+\sqrt{2} \gs \dd B_t\\
\dd \tilde Y_t = \frac{\gd}{c}\left(\tilde X_t-b\tilde Y_t\right)\dd t
\end{array}
\right. ,
\end{equation} 
where $m_t=(x_t,y_t)=(\bE[X_t],\bE[Y_t])$. Observe that, by \eqref{eq:kinetic FhN}, $m_{ t}$ solves 
\begin{equation}
\label{eq:evol_mt}
\left\{\begin{array}{l}
\dot x_{ t} = \gd \left( x_{ t}-\frac{ \int x^{ 3} {\rm d}\mu_{ t}}{3} - y_{ t}\right)\\
\dot y_{ t} = \frac{\gd}{c}\left(x_{ t}+a-b y_{ t}\right)
\end{array}
\right. .
\end{equation} 
\begin{lemma}\label{lem:close to q0}
There exists positive constants $\gd_1$ and $ \kappa_1$ such that, for all initial condition $ \mu_{ 0}$ such that $ m_{ 0}=\left(x_{ 0}, y_{ 0}\right)=\left( \int x \mu_{ 0}( {\rm d}x, {\rm d}y), \int y \mu_{ 0}({\rm d}x, {\rm d}y)\right)\in B \left(0, R_{ 0}\right)$ and $ \int x^{ 6} \mu_{ 0}({\rm d}x, {\rm d}y)\leq \kappa_{ 0}^{ x}$ and $ \int y^{ 6} \mu_{ 0}({\rm d}x, {\rm d}y)\leq \kappa_{ 0}^{ y}$ (recall Lemma \ref{lem:moments X Y}), the following is true: if $\gd\leq \gd_1$ and $W(\tilde \mu_0,q_{0})\leq \kappa_1\gd $, then
\begin{equation}
\sup_{t\in[0,T_e(R_0)]} W(\tilde \mu_t,q_{0})\leq \kappa_1\gd.
\end{equation}
\end{lemma}

\begin{proof}
We suppose that $ \delta\leq \delta_{ 0}$ where $ \delta_{ 0}$ is given by Lemma~\ref{lem:moments X Y}. For an arbitrary $ \varepsilon>0$, consider a coupling $ \tilde{\pi}_{ 0}\in \mathcal{ C}( \tilde{ \mu}_{ 0}, q_{ 0})$ such that
\begin{equation}
\label{eq:coupling_0}
 \Lambda( \tilde{ \pi}_{ 0})^{ 2} < W \left( \tilde{ \mu}_{ 0}, q_{ 0}\right)^{ 2} + \varepsilon.
\end{equation}
In the following, we consider $ \left\lbrace \left( \tilde{ X}_{ 0}, \tilde{ Y}_{ 0}\right), \left( Z_{ 0}^{ x}, Z_{ 0}^{ y}\right)\right\rbrace$ with law $ \tilde{\pi}_{ 0}$. For this initial condition, we consider both $(X_t,Y_t)$ solution to \eqref{eq:kinetic FhN} with initial condition $(X_{ 0}, Y_{ 0})=(\tilde X_0 + x_{ 0},\tilde Y_0+ y_{ 0})$ and $(\tilde X_t,\tilde Y_t)$, solution to \eqref{eq:evol_tildeXY} with initial condition $(\tilde X_0,\tilde Y_0)$. Consider also the process $( Z^x_t,Z^y_t)$ solution to the system \eqref{eq: def Z} with initial condition $(Z^x_0,Z^y_0)$ with distribution $q_0$. The calculations made in Section~\ref{sec:slow-fast approach} show that $(Z^x_t,Z^y_t)$ has also distribution $q_0$, so that, by construction, the law $ \tilde{ \pi}_{ t}$ of $ \left\lbrace \left( \tilde{ X}_{ t}, \tilde{ Y}_{ t}\right), \left( Z_{ t}^{ x}, Z_{ t}^{ y}\right)\right\rbrace$ belongs to $ \mathcal{ C}( \tilde{ \mu}_{ t}, q_{ 0})$. We have 
\begin{equation}
\frac12 \dd (\tilde X_t-Z^x_t)^2 =-K(X_t-Z^x_t)^2\dd t+ (\tilde X_t-Z^x_t) \left(\gd \left( X_t-\frac{ X_t^3}{3} - Y_t\right)-\dot x_t \right)\dd t.
\end{equation}
By Lemma \ref{lem:moments X Y} and \eqref{eq:evol_mt}, there exists a positive constant $C_{ 0}$ such that for all $t\in \left[0,T_e(R_0)\right]$,
\begin{equation}
\bE\left[\left(\gd \left( X_t-\frac{ X_t^3}{3} - Y_t\right)- \dot x_t \right)^2\right]^{\frac12}\leq C_{ 0} \gd\, .
\end{equation}
Note that the constant $C_{ 0}$ only depends on the parameters $(a, b, c)$ of the model and on $R_{ 0}$ through the constants $ \kappa_{ 0}^{ x}$ and $ \kappa_{ 0}^{ y}$ defined in Lemma~\ref{lem:moments X Y}. In particular, $C_{ 0}$ does not depend on $ \varepsilon$. We then have, for all $t\in \left[0,T_e(R_0)\right]$,
\begin{equation}
\frac12 \frac{\dd}{\dd t}\bE[(\tilde X_t-Z^x_t)^2]\leq C_{ 0} \gd (\bE[(\tilde X_t-Z^x_t)^2])^{\frac12}
- K \bE[(\tilde X_t-Z^x_t)^2].
\end{equation}
If now we choose the constant $ \kappa_{ 1}$ to be $ \kappa_1:=\frac{C_{ 0}}{K} $ and assume that $W(\tilde \mu_0,q_{0})\leq \kappa_1\gd $, then by 
\eqref{eq:coupling_0}, we have $\bE[(\tilde X_0-Z^x_0)^2]\leq \frac{C_{ 0}^2}{K^2}\gd^2 + \varepsilon$. Consequently, by Lemma~\ref{lem:gronwall_sqrt},
\begin{equation}
\label{eq:control_X_Zx}
\sup_{t\in [0,T_e(R_0)]}\bE[(\tilde X_t-Z^x_t)^2]\leq \frac{C_{ 0}^2}{K^2}\gd^2 + \varepsilon.
\end{equation}
In a same way,
\begin{equation}
\frac{c}{2\gd} \dd (\tilde Y_t-Z^y_t)^2= (\tilde Y_t-Z^y_t)(\tilde X_t-Z^x_t)\dd t - b (\tilde Y_t-Z^y_t)^2\dd t,
\end{equation}
and from the estimate \eqref{eq:control_X_Zx} we have just obtained above, we get for $t\in \left[0,T_e(R_0)\right]$
\begin{equation}
\frac{c}{2\gd} \frac{\dd}{\dd t}\bE[(\tilde Y_t-Z^y_t)^2]\leq \left( \frac{ C_{ 0}^{ 2} \delta^{ 2}}{ K^{ 2}}+ \varepsilon\right)^{ \frac{ 1}{ 2}} (\bE[(\tilde Y_t-Z^y_t)^2])^{\frac12} - b \bE[(\tilde Y_t-Z^y_t)^2].
\end{equation}
From $W(\tilde \mu_0,q_{0})\leq \kappa_1\gd $, we obtain that $\bE[(\tilde Y_0-Z^y_0)^2]\leq \frac{ 1}{ b^{ 2}} \left(\frac{C_{ 0}^2}{K^2}\gd^2 + \varepsilon\right)$, so that by Lemma~\ref{lem:gronwall_sqrt}, we have 
\begin{equation}
\label{eq:control_Y_Zy}
\sup_{t\in [0,T_e(R_0)]}\bE[(\tilde Y_t-Z^y_t)^2]\leq \frac{ 1}{ b^{ 2}} \left(\frac{C_{ 0}^2}{K^2}\gd^2 + \varepsilon\right)\, .
\end{equation} 
Estimates \eqref{eq:control_X_Zx} and \eqref{eq:control_Y_Zy} are a fortiori true for the infimum over all couplings $W( \tilde{ \mu}_{ t}, q_{ 0})$. Letting $ \varepsilon \searrow 0$, this concludes the proof, taking $\gd_1=\gd_0$.
\end{proof}
\section{Floquet Theory and contraction close to $\gamma$}\label{sec:contract}
We introduce in this paragraph the minimal notions from Floquet theory \cite{teschl2012ordinary} that are necessary for our purpose. For the simplicity of exposition, this is first done in the case $ \delta=1$ from which the general case $ \delta>0$ may be deduced up to minor modifications (see section \ref{sec:def_general_delta} below).
\subsection{The case $ \delta=1$}
Recall the definition of $F_{ u}$ in \eqref{eq:FHN_u}. Denote by $\Pi(s,t)$ the principal matrix solution associated to the periodic solution $\gamma= \gamma^{ 1}$ of \eqref{eq:FhN gs2/K}, when $ \delta=1$, that is the solution to
\begin{equation}
\partial_t \Pi(t,s)= \, D F_{\frac{\gs^2}{K}}(\gamma_t) \Pi(t,s), \quad \Pi(s,s)=I_d.
\end{equation}
$\Pi(t,s)$ is invertible, with
\begin{equation}
\Pi(t,s)\Pi(s,t)=I_d,
\end{equation}
and satisfies (recall that $t\mapsto \gamma_t$ has period $T_\gamma$):
\begin{equation}
\Pi(t,u)\Pi(u,s)=\Pi(t,s), \quad \Pi(t+T_\gamma,s+T_\gamma)=\Pi(t,s).
\end{equation}
Denote $Q(0)$ the matrix such that
\begin{equation}
\Pi(T_\gamma,0)=e^{-T_\gamma Q(0)}.
\end{equation}
Since
\begin{equation}
F_{\frac{\gs^2}{K}}(\gamma_t)=\Pi(t,s)F_{\frac{\gs^2}{K}}(\gamma_s),
\end{equation}
one of the eigenvalues of $Q(s)$ is simply $0$, with corresponding eigenvector $F_{\frac{\gs^2}{K}}(\gamma_0)$. By hypothesis $\gamma$ is stable and thus the other eigenvalue is positive (see \cite{teschl2012ordinary}, Chapter 12), we denote it $\lambda$ and $e_0$ a corresponding eigenvector. Define $N(t,0)=\Pi(t,0)e^{ t Q(0)}$, so that $P(t,0)=N(t,0)e^{- tQ(0)}$. $N(t,0)$ is clearly periodic: $N(t+T_\gamma,0)=N(t,0)$.

We can now define
\begin{equation}
Q(s)=\Pi(s,0)Q(0)\Pi(0,s),\quad \text{and} \quad N(t,s)=\Pi(t,t-s) N(t-s,0) \Pi(0,s),
\end{equation}
so that
\begin{equation}
\Pi(t,s)=N(t,s)e^{-(t-s) Q(s)}, 
\end{equation}
and $N$ satisfies $N(t+T_\gamma,s)=N(t,s)$ and $Q(s)$ has the same spectral decomposition as $Q(0)$. In particular $e_s=N(s,0)e_0$ is an eigenvector associated with the eigenvalue $\gl$ for $Q(s)$. It is easy to see that $Q(t)=\Pi(t,s)Q(s)\Pi(s,t)$ and $N(t,u)N(u,s)=N(t,s)$.

We consider $p_s$ and $p^\perp_s$ the linear form associated to the projection on $F_{\frac{\gs^2}{K}}$ and $e_s$ respectively, i.e. $p_s,p_s^\perp:\bbR^2\rightarrow\bbR$ and satisfies $p_s(e_s)=p_s^\perp(F_{\frac{\gs^2}{K}}(\gamma_s))=0$ and $p_s(F_{\frac{\gs^2}{K}}(\gamma_s))=p_s^\perp(e_s)=1$. It is easy to see that $p^\perp_t(N(t,s)\cdot)=p^\perp_s(\cdot)$ and that, for all $t>s$,
\begin{equation}\label{eq:contract p perp_1}
\left|p_t^\perp\left(\Pi(t,s)u\right)\right|\leq e^{-\gl (t-s) }p^\perp_s(u).
\end{equation}
We consider also a constant $C_\Pi$ such that $|\Pi(t,s)u|\leq C_\Pi |u|$ for all $t>s$.

\medskip

For $\ga>0$ and all $s\in\left[0,T_\gamma\right)$, we define $E_s(\ga)$ as
\begin{equation}
E_s(\ga):=\{ u\in \gamma_s+\text{span}(e_s):\, \left| p^\perp_s(u-\gamma_s)\right|\leq \ga\}.
\end{equation} 

\subsection{ The general case $ \delta>0$}
\label{sec:def_general_delta}
All that have been exposed in the previous paragraph for the limit cycle $ \gamma^{ 1}$ in the case $ \delta=1$ can be easily transposed for $ \gamma^{ \delta}$ when $ \delta>0$, with the following definitions: $ \Pi^{ \delta}(t, s) := \Pi^{ 1}(\delta t, \delta s)$, $ Q^{ \delta}(s):= \delta Q^{ 1}( \delta s)$, $N^{ \delta}(t, s)= N^{ 1}( \delta t, \delta s)$, $ e_{ s}^{ \delta}:= e_{  \delta s}^{ 1}$, $p_{ s}^{ \delta}:= p^{ 1}_{ \delta s}$, $p_{ s}^{ \delta, \perp}:= p_{ \delta s}^{ 1, \perp}$ and
\begin{equation}
E^{ \delta}_s(\ga):=\{ u\in \gamma^{ \delta}_s+\text{span}(e^{ \delta}_s):\, \left| p^{ \delta, \perp}_s(u-\gamma^{ \delta}_s)\right|\leq \ga\},\ s\in[0, T_{ \gamma}/ \delta)\, .
\end{equation} 

Note in particular the spectral gap $ \lambda$ when $ \delta=1$ is changed into $ \delta \lambda$: \eqref{eq:contract p perp_1} becomes
\begin{equation}\label{eq:contract p perp_delta}
\left|p_t^{ \delta, \perp}\left(\Pi^{ \delta}(t,s)u\right)\right|\leq e^{-\gl \delta(t-s) }p^{ \delta, \perp}_s(u).
\end{equation}

We now give a classical result of projection $\gamma$. For a similar result in a more general situation, see for example \cite{sell2013dynamics}. 

\begin{lemma}\label{lem:proj}
There exist $\ga_0>0$ such that for all $z$ in the $\ga_0$-neighborhood of $\gamma^{ \delta}$, there exists a unique $\theta=:\proj^{ \delta}(z) \in \mathbb{ S}_{ \delta}$ such that $z\in E^{ \delta}_\theta(1)$. 
Moreover, for all $z, h$ such that $z$ and $z+h$ are in the $ \alpha_{ 0}$-neighborhood of $ \gamma^{ \delta}$, if $ \theta^{ \delta}= \proj^{ \delta}(z)$,
\begin{equation}
\label{eq:expansion proj}
\proj^{ \delta}(z+h)=\theta^{ \delta}+ \frac{ 1}{ \delta}\frac{p^{ \delta}_{\theta^{ \delta}}(h)}{1+p^{ \delta}_{\theta^{ \delta}}\left(\left(DF_{\frac{\gs^2}{K}}(\gamma^{ \delta}_{\theta^{ \delta}})+ \frac{ 1}{ \delta}Q^{ \delta}(\theta^{ \delta})\right)(z-\gamma^{ \delta}_{\theta^{ \delta}})\right)}+ \frac{ 1}{ \delta}O(h^2)\ ,
\end{equation}
where the rest $O(h^2)$ is uniform in $ \delta\leq 1$. In particular, under the previous hypotheses, there exists a constant $C_{ \proj}>0$, independent of $ \delta$, such that
\begin{equation}
\label{eq:proj_lip}
d_{ \mathbb{ S}_{ \delta}} \left(\proj^{ \delta}(z+h), \proj^{ \delta}(z)\right)\leq \frac{ C_{ \proj}}{ \delta} \left\vert h \right\vert\, .
\end{equation}
\end{lemma}
\begin{proof} Assume first that the result holds in the case $ \delta=1$. Then it is easy to see that the unique candidate for the projection in the case $ \delta>0$ is given by $ \proj^{ \delta}(z):= \frac{ \proj^{ 1}(z)}{ \delta}\in \mathbb{ S}_{ \delta}$. In particular, one deduces \eqref{eq:expansion proj} and \eqref{eq:proj_lip} from the case $ \delta=1$ and the change of variables formulas in \S~\ref{sec:def_general_delta}. We now prove the result for $ \delta=1$, which basically follows from the Implicit Functions Theorem. Let 
\begin{equation}
(z, \theta)\in \mathbb{ R}^{ 2}\times \mathbb{ S}_{ \delta}\mapsto f(z, \theta):= p_{ \theta} \left(z- \gamma_{ \theta}\right)\, .
\end{equation}
Relying on the identities $p_\theta(z-\gamma_\theta)=p_0(N(0,\theta)(z-\gamma_\theta))$ and
\begin{multline}
\frac{\dd}{\dd \theta}N(0,\theta)=\frac{\dd}{\dd \theta} (N(\theta,0)^{-1})=- N(0,\theta) \left(DF_{\frac{\gs^2}{k}}(\gamma_\theta)N(\theta,0)+N(\theta,0)Q(0) \right) N(0,\theta)
\\=-N(0,\theta) \left( DF_{\frac{\gs^2}{K}}(\gamma_\theta)+ Q(\theta)\right),
\end{multline}
we obtain
\begin{multline}
\partial_\theta f(z, \theta)=-p_\theta\left(F_{\frac{\gs^2}{k}}(\gamma_\theta)\right)- p_\theta\left(\left( DF_{\frac{\gs^2}{K}}(\gamma_\theta)+ Q(\theta)\right)(z-\gamma_\theta) \right)\\
= -1 -  p_\theta\left(\left( DF_{\frac{\gs^2}{K}}(\gamma_\theta)+ Q(\theta)\right)(z-\gamma_\theta) \right).
\end{multline}
Taking $z_0=\gamma_\theta +\ga e_\theta$ with $\ga$ small enough, we have $f(z_0,\gamma_\theta)=0$ and $\partial_\theta f(z_0,\gamma_\theta)\neq 0$, so that the existence, local uniqueness and smoothness of the projection $\proj(z)$ in a tubular neighborhood of $\gamma$ follows from the Implicit functions Theorem and the compactness of $\gamma$. Moreover, denoting $\theta=\proj(z)$ and $\theta_h=\proj(z+h)-\proj(z)$, we have $\theta_h=O(h)$ by smoothness of the projection, and
\begin{multline}
p_{\theta+\theta_h}(z+h-\gamma_{\theta+\theta_h})=0\\
=p_\theta(z+h-\gamma_\theta)-\theta_hp_\theta(\dot \gamma_\theta)-\theta_h p_\theta\left(\left(DF_{\frac{\gs^2}{K}}(\gamma_\theta)+ Q(\theta)\right)(z-\gamma_\theta)\right)+O(h^2)\\
= p_\theta(h)-\theta_h\left( 1+p_\theta\left(\left(DF_{\frac{\gs^2}{K}}(\gamma_\theta)+ Q(\theta)\right)(z-\gamma_\theta)\right)\right)+O(h^2),
\end{multline}
which implies \eqref{eq:expansion proj}.
\end{proof}
In the following $ \gamma= \gamma^{ \delta}$ stands for the limit cycle for $ \delta>0$. We will denote
\begin{equation}
\dist_{\Pi}^{ \delta}(z,\gamma):=\left\vert p^{ \delta, \perp}_{\proj(z)}\left( z-\gamma_{\proj(z)}\right)\right\vert.
\end{equation}
Now a variation of constants (see \cite{teschl2012ordinary}, page 84) shows that the solution to
\begin{equation}
\dot z_t=\gd DF_{\frac{\gs^2}{K}}(\gamma_{\theta_0+t})z_{ t} + g(t),
\end{equation} 
for a smooth mapping $g(t)$, satisfies
\begin{equation}
\label{eq:var_const_z}
z_{t}= \Pi^{ \delta}(\theta_0+ t,\theta_0)z_{t}+\int_{0}^t \Pi^{ \delta}(\theta_0+ t,\theta_0+s)g(s)\dd s.
\end{equation}
So, since
\begin{equation}
\dot m_t-\dot \gamma_{\theta_0+t}=\gd \int_{\bbR^2} F(z+m_t)\dd \tilde \mu_t(z)-\gd \int_{\bbR^2} F(z+\gamma_{\theta_0+t})q_0(z)\dd z,
\end{equation}
we have
\begin{equation}\label{eq:mild mt-gammat}
m_t-\gamma_{ \theta_{ 0}+t}=\Pi^{ \delta}(\theta_0+t,\theta_0)(m_0-\gamma_0)+\gd \int_0^t \Pi^{ \delta}(\theta_0+t,\theta_0+s) g(s),
\end{equation}
with
\begin{multline}
g(s)=F_{\frac{\gs^2}{K}}(m_s)-F_{\frac{\gs^2}{K}}(\gamma_{\theta_0+s})-DF_{\frac{\gs^2}{K}}(\gamma_{\theta_0+s})(m_s-\gamma_{\theta_0+s})\\+\int_{\bbR^2}F(z+m_s)\dd \tilde  \mu_s(z)-\int_{\bbR^2}F(z+m_s)q_0(z)\dd z.
\end{multline}

For an initial condition $\mu_0$ of \eqref{eq:PDER kinetic FhN} with $m_0=\int_{\bbR^2}z\dd \mu_0(z)\in E^{ \delta}_\theta(1)$, we define $T_{r,\theta}(\mu_0)$ the return time
\begin{equation}
T_{r,\theta}(\mu_0)=\inf\{t>0:\, m_t\in E^{ \delta}_\theta(1)\}.
\end{equation}

\begin{lemma}\label{lem:proj on gamma}
There exist positive constants $\kappa_2$ and $\gd_2$ such that for any $ \mu_{ 0}$ satisfying the hypotheses of Lemma \ref{lem:moments X Y} and Lemma \ref{lem:close to q0}, the following is true: if $\gd\leq \gd_2$ and $\dist^{ \delta}_\Pi(m_0,\gamma)\leq \kappa_2\gd$, then
\begin{equation}
\sup_{t\geq 0} \dist^{ \delta}_\Pi(m_t,\gamma)\leq \kappa_2\gd.
\end{equation}
In particular $T_e(R_0)=\infty$.
\end{lemma}

\begin{proof}
First remark that, by regularity of $F$ and compactness of the limit cycle $ \gamma$, there exists positive constants $C_1$ and $\ga_F$ such that for all $s\in \left[0,T_\gamma/ \delta\right]$ and $m$ such that $|m-\gamma_s|\leq \ga_F$, we have
\begin{equation}
\label{eq:diffF}
\left|F_{\frac{\gs^2}{K}}(m)-F_{\frac{\gs^2}{K}}(\gamma_s)-DF_{\frac{\gs^2}{K}}(\gamma_s)(m-\gamma_s) \right|\leq C_1 |m-\gamma_s|^2.
\end{equation}
Consider a $\kappa_2>0$ and suppose that $\gd$ is small enough such that $(1+ C_{\Pi})\kappa_2 \gd < \ga_F$. Suppose then that $m_0\in E^{ \delta}_{\theta_0}(\kappa_2\gd)$ for $ \theta_0= \proj(m_{ 0})\in \mathbb{ S}_{ \delta} $ and consider the time
\begin{equation}
t_1:=\inf\{t\geq 0:\, |m_t-\gamma_{\theta_0+ t}|\geq (1+ C_{ \Pi}) \kappa_2 \gd \}.
\end{equation}
Since $\gamma$ is strictly included in $B(0,R_0)$, for $\gd $ taken small enough we have $ t_1\leq T_e(R_0)$. For $t\leq t_1$ we obtain
\begin{equation}
|m_t-\gamma_{\theta_0+t}|\leq C_\Pi |m_0-\gamma_{\theta_0}|+ C_\Pi \gd \int_0^t |g(s)|\dd s,
\end{equation}
and
\begin{equation}
|g(s)|\leq C_{ 1}|m_s-\gamma_{\theta_0+ s}|^2 +\left|\int_{\bbR^2} F(z+m_s)\dd \tilde \mu_s(z)-\int_{\bbR^2}F(z+m_s)q_0(z)\dd z\right|.
\end{equation}
Now, recalling the notations used in the proof of Lemma \ref{lem:close to q0}, we have
\begin{multline}\label{eq:decomp diff integrales}
\left|\int_{\bbR^2} F(z+m_s)\dd \tilde \mu_s(z)-\int_{\bbR^2}F(z+m_s)q_0(z)\dd z\right|^2\\=\left( \bE\left[X_s-\frac{X_s^3}{3}-Y_s\right]
-\bE\left[Z^x_s+ x_s-\frac{(Z^x_s+x_s)^3}{3}-Z^y_s-y_s\right]\right)^2\\
+\left( \bE\left[X_s+a-bY_s\right]-\bE\left[Z^x_s+x_s+a-Z^y_s-y_s\right]\right)^2\, ,
\end{multline}
where $ \left\lbrace \left(X_{ t}, Y_{ t}\right), \left(Z_{ t}^{ x}, Z_{ t}^{ y}\right)\right\rbrace$ is the $ \varepsilon$-coupling introduced in Lemma~\ref{lem:close to q0}.
First, note that, by \eqref{eq:control_X_Zx}, we have for $t\leq t_{ 1}\leq T_{ e}(R_{ 0})$, $\bE[X_t-x_t-Z^x_t]^2 \leq \mathbf{ E} \left[ \left\vert \tilde X_{ t} - Z_{ t}^{ x}\right\vert^{ 2}\right]\leq  \kappa_1^2\gd^2 + \varepsilon$.  Moreover, using Lemma~\ref{lem:moments X Y} and Lemma~\ref{lem:close to q0}, 
\begin{align*}
\bE[X_t^3-(Z^x_t+x_t)^3]&= \mathbf{ E} \left[ \tilde X_{ t}^{ 3} - \left(Z_{ t}^{ x}\right)^{ 3}\right]+3x_{ t}\mathbf{ E}\left[ \tilde X_{ t}^{ 2} - \left(Z_{ t}^{ x}\right)^{ 2}\right],\\
&\leq C_{R_0,\gs^2/K} \mathbf{ E} \left[ \left\vert \tilde X_{ t} - Z_{ t}^{ x} \right\vert^{ 2}\right]^{ \frac{ 1}{ 2}}\leq C_2 \left( \delta^{ 2}+ \varepsilon\right)^{ \frac{ 1}{ 2}},
\end{align*}
for some constant $C_2>0$, independent of $ \varepsilon$. This altogether implies that, for some constant $C_3>0$,
\begin{equation}
|g(s)|\leq C_1|m_s-\gamma_{\theta_0+s}|^2 +C_3 \left( \delta^{ 2}+ \varepsilon\right)^{ \frac{ 1}{ 2}}.
\end{equation}
Taking $ \varepsilon \searrow 0$ ($g(s)$ does not depend on $ \varepsilon$), one obtains
\begin{equation}
|g(s)|\leq C_1|m_s-\gamma_{\theta_0+s}|^2 +C_3 \delta.
\end{equation}
We deduce that, for $t\leq t_1 $, since $m_{ 0}- \gamma_{ \theta_0}= p_{ \theta_0}^{ \perp}(m_{ 0}- \gamma_{ \theta_0})$,
\begin{equation}\label{eq:first bound m-gamma}
|m_t-\gamma_{\theta_0+ t}|\leq C_{ \Pi}\kappa_2 \gd   \left(1+ \gd t \left(C_1(1+C_{ \Pi})^{ 2}\kappa_2 \gd + \frac{C_3}{\kappa_2}\right)\right),
\end{equation}
and thus, choosing $\kappa_2 \geq 4C_3C_\Pi T_\gamma$ and $\gd\leq \frac{1}{4T_\gamma C_1\kappa_2 C_{ \Pi}(1+C_{ \Pi})^{ 2}}$ we obtain that 
\begin{equation}
\label{eq:t1_T}
t_1 \geq \frac{2 T_\gamma}{\gd}\, .
\end{equation}
Note that choosing also $ \delta \leq \frac{ \alpha_{ 0}}{ (1+ C_{ \Pi}) \kappa_{ 2}}$, (recall the definition of $ \alpha_{ 0}$ in Lemma~\ref{lem:proj}) ensures that $m_{ t}$ stays in the $ \alpha_{ 0}$-neighborhood of the limit cycle $ \gamma$. In particular, the projection 
\begin{equation}
\theta_t=\theta^\gd_t=\proj^\gd(m_t)
\end{equation}
is well-defined for all $t\in [0, \frac{ 2 T_{ \gamma}}{ \delta})$.

Let us now prove that $\dot \theta^\gd_t=1+O(\gd)$ for $t\leq t_1$, which implies in particular that $T_{r,\theta}(\mu_0)\in \left[\frac{T_\gamma}{\gd}-c_1,\frac{T_\gamma}{\gd}+c_1\right]$ for some $c_1>0$ that does not depend on $\gd$ (but may depend on $\kappa_2$). Now, recalling Lemma \ref{lem:proj}, for $s\leq t_1$ we have
$\dot\theta_s=\frac{1}{\gd}p_{\theta_s}(\dot m_s)+O(\gd)$, with
\begin{equation}
p_{\theta_s}(\dot m_s)= \delta+p_{\theta_s}(\dot m_s)-p_{\theta_s}\left(\delta F_{\frac{\gs^2}{K}}(\gamma_{\theta_s})\right),
\end{equation}
and
\begin{align}
\frac{ 1}{ \delta} \left(\dot m_s- \delta F_{\frac{\gs^2}{K}}(\gamma_{\theta_s})\right)
&=F_{\frac{\gs^2}{K}}(\gamma_{\theta_0+ s})- F_{\frac{\gs^2}{K}}(\gamma_{\theta_s})+ F_{\frac{\gs^2}{K}}(m_s)-F_{\frac{\gs^2}{K}}(\gamma_{\theta + s})\nonumber\\
&\qquad \qquad + \int_{\bbR^2}F(z+m_s)\dd \tilde \mu_s(z)- \int_{\bbR^2} F(z+m_s)q_0(z)\dd z\nonumber\\
&= F_{\frac{\gs^2}{K}}(\gamma_{\theta_0+ s})- F_{\frac{\gs^2}{K}}(\gamma_{\theta_s})+O(\gd),
\end{align}
where we have used the fact that $t\leq t_1$ and $F_{\frac{\gs^2}{K}}$ is smooth, and similar arguments as above to tackle the difference of integral terms.
Moreover, by \eqref{eq:proj_lip} and the definition of $ \gamma=\gamma^{ \delta}$, $\vert\gamma_{\theta_0 + s}-\gamma_{\theta_s}\vert=\vert\gamma^{ 1}_{ \delta(\theta_0 + s)}-\gamma^{ 1}_{ \delta \theta_s}\vert=\vert \gamma^{ 1}_{ \proj^{1}(\gamma_{\theta_0+s})}-\gamma^{ 1}_{ \proj^{1}(m_s)}\vert \leq C_{ \proj, \gamma}|m_s-\gamma_{\theta_0+ s}|$ for some constant $C_{ \proj, \gamma}$. By \eqref{eq:t1_T}, this last quantity is of order $O( \kappa_{ 2} \delta)$, uniformly in $s\in [0, \frac{ 2T_{ \gamma}}{ \delta})$. We obtain $p_{\theta_s}(\dot m_s)=\gd +O(\gd^2)$, and thus $\dot \theta_s=1+O(\gd)$.

It remains now to prove that $|m_{t}-\gamma_{\proj(m_t)}|\leq \kappa_2 \delta$ for all $t\leq T_{r,\theta}(\mu_0)$, taking a larger value for $\kappa_2$ if needed. But we have on one hand, recalling \eqref{eq:contract p perp_delta}, \eqref{eq:mild mt-gammat} and the estimates made above,
\begin{equation}
\left|p_{\theta_0+ t}^\perp\left(m_{t}-\gamma_{\theta_0+t}\right)\right|\leq \kappa_2 \gd \left( e^{-\gl \gd  t}+C_\Pi \gd t\left(C_1(1+C_{ \Pi})^{ 2}\kappa_2 \gd +\frac{C_3}{\kappa_2} \right)\right).
\end{equation}
Now, remarking that, for $ \eta = \frac{1-e^{-2\gl T_\gamma}}{2T_\gamma}$, $e^{-\gl u}+ \eta u\leq 1$ for $u\in [0,2T_\gamma]$, we have
\begin{equation}\label{eq: first bound proj mt}
\left|p_{\theta_0+ t}^\perp\left(m_{t}-\gamma_{\theta_0+t}\right)\right|\leq \left(1-\frac{ \eta}{2} \gd t\right)\kappa_2 \gd, 
\end{equation}
if we take $\kappa_2\geq \frac{4C_3C_\Pi}{ \eta}$ and $\gd \leq \frac{ \eta}{4 C_1 C_\Pi(1+ C_{ \Pi})^{ 2} \kappa_2}$.

On the other hand, using the fact that $p^\perp_{\theta_t}(\dot \gamma_{\theta_t})=0$, that for all $w$ 
\begin{equation*}
\left\vert\left( p^{ \delta, \perp}_{\theta_t}-p^{ \delta, \perp}_{\theta_0+t}\right)(w)\right\vert=\left\vert\left( p^{ 1, \perp}_{ \delta \theta_t}-p^{ 1, \perp}_{ \delta\theta_0+ \delta t}\right)(w)\right\vert \leq \delta C_{\proj,\gamma} \left\vert w \right\vert \left\vert \theta_0 + t - \theta_{ t} \right\vert,
\end{equation*} and the fact that $|m_t-\gamma_{\theta_0+ t}|\leq (1+C_{ \Pi})\kappa_2\gd$, we have 
\begin{multline}\label{eq: second bound proj mt}
\left\vert p^\perp_{\theta_t}(m_t-\gamma_{\theta_t})-p^\perp_{\theta_0+ t}(m_t-\gamma_{\theta_0+ t})\right\vert \leq \left\vert p^\perp_{\theta_t}(\gamma_{\theta_0+t}-\gamma_{\theta_t})\right\vert+\left\vert\left( p^\perp_{\theta_t}-p^\perp_{\theta_0+t}\right)(m_t-\gamma_{\theta_0+t})\right\vert\\
\leq C_{\proj,\gamma} \delta^{ 2}\left(\vert \theta_0+ t-\theta_t\vert ^2 + \kappa_2 \vert \theta_0+ t-\theta_t\vert  \right).
\end{multline}
If $v_t:= \theta_t-(\theta_0+ t)$ we have, using \eqref{eq:expansion proj}
\begin{align}
\delta\dot v_{ t} - p_{ \theta_{ t}}(\dot m_{ t}) + \delta&= p_{ \theta_{ t}}(\dot m_{ t}) \left(\frac{ 1}{ 1+ p_{ \theta_{ t}} \left( \left(DF_{  \frac{ \sigma^{ 2}}{ K}}( \gamma_{ \theta_{ t}}) + \frac{ 1}{ \delta}Q^{ \delta}(\theta_{ t})\right)(m_{ t} - \gamma_{ \theta_{ t}})\right)} - 1\right)
\end{align}
The term within the brackets is controlled by $C \left\vert m_{ t} - \gamma_{ \theta_{ t}}\right\vert$, for some universal constant $C>0$. By the regularity of $ t \mapsto \gamma^{ \delta}_{ t}= \gamma^{ 1}_{ t \delta}$ and the Lipchitz-continuity of $\proj(\cdot)$ (recall \eqref{eq:proj_lip}), there is a constant $C_{ \gamma, \proj}>0$
\begin{align*}
\left\vert \gamma_{ \theta_0 + t} - \gamma_{ \theta_{ t}}\right\vert & = \left\vert \gamma_{ \proj( \gamma_{ \theta_0+t})} - \gamma_{ \proj(m_{ t})}\right\vert \leq C_{  \gamma, \proj} \left\vert \gamma_{ \theta_0 + t} - m_{ t}\right\vert.
\end{align*} 
Hence, $\left\vert m_{ t} - \gamma_{ \theta_{ t}}\right\vert \leq \left\vert m_{ t} - \gamma_{ \theta_0 + t}\right\vert + \left\vert \gamma_{ \theta_0 + t} - \gamma_{ \theta_{ t}}\right\vert \leq C^{ \prime}_{ \gamma, \proj} \left\vert m_{ t} - \gamma_{ \theta_0+t}\right\vert$. Putting things together with $|m_t-\gamma_{\theta_0+ t}|\leq (1+C_{ \Pi})\kappa_2\gd$ (recall \eqref{eq:t1_T}), we obtain finally that
\begin{equation}
\vert \delta\dot v_t -p_{\theta_t}(\dot m_t)+\gd\vert \leq C_{\proj,\gamma} \kappa_2 \gd \vert p_{\theta_t}(\dot m_t)\vert .
\end{equation}

So recalling that $p_{\theta_t}(\dot m_t)=\gd +O(\kappa_2\gd^{ 2})$, we have $\dot v_t=O(\kappa_2\gd)$, and thus $v_t=tO(\kappa_2 \gd)$, since $v_{ 0}=0$. This means that for $t\in \left[0,\frac{2T_\gamma}{\gd}\right]$ and some positive constant $C_4$,
\begin{equation}\label{eq: third bound proj mt}
\left\vert p^\perp_{\theta_t}(m_t-\gamma_{\theta_t})-p^\perp_{\theta_0+t}(m_t-\gamma_{\theta_0+t})\right\vert \leq C_4\kappa_2^2\gd^3 t.
\end{equation}
Finally, we obtain, combining \eqref{eq: first bound proj mt} and \eqref{eq: third bound proj mt},
\begin{equation}
\dist_\Pi(m_t,\gamma)\leq \left(1-\frac{ \eta}{2}\gd t+C_4 \kappa_2 \gd^2 t\right)\kappa_2 \gd,
\end{equation}
which implies that $|m_{t}-\gamma_{\proj(m_t)}|\leq \kappa_2 \delta$ for all $t\leq T_{r,\theta}(\mu_0)$, taking $\gd\leq \frac{ \eta}{2C_4\kappa_2}$. This completes the proof, since $\mu_{T_{r,\theta}(\mu_0)}$ satisfies then the same hypotheses as $\mu_0$ (recall Lemma \ref{lem:moments X Y}, Lemma \ref{lem:close to q0} and the fact that $ T_{r,\theta}(\mu_0)\leq T_e(R_0)$), the result for all $t$ is obtained by recursion.
\end{proof}

For two solutions $(X_{1,t},Y_{1,t})$, $(X_{2,t},Y_{2,t})$ of \eqref{eq:kinetic FhN} with initial conditions $(X_{1,0},Y_{1,0})$ and $(X_{2,0},Y_{2,0})$, with respective distributions $\mu_{1,t}$ and $\mu_{2,t}$ and expectations $m_{1,t}=(x_{ 1, t}, y_{ 1, t})$ and $m_{2,t}=(x_{ 2, t}, y_{ 2, t})$, we will use the notations
\begin{multline}
\Delta X_t=X_{1,t}-X_{2,t},\qquad \Delta Y_t=Y_{1,t}-Y_{2,t}, \qquad \Delta \tilde X_t=\tilde X_{1,t}-\tilde X_{2,t}, \qquad \Delta \tilde Y_t=\tilde Y_{1,t}-\tilde Y_{2,t},\\
\Delta m_t=m_{1,t}-m_{2,t},\qquad \Delta \theta_t = \theta_{1,t}-\theta_{2,t}=\proj(m_{1,t})-\proj(m_{2,t}).
\end{multline}
We choose $T$ such that
\begin{equation}\label{hyp T}
e^{-\gl T}\leq \frac{1}{4}, \quad e^{-\frac{b}{c}T}\leq \frac{1}{16}.
\end{equation}

\begin{lemma}\label{lem:bound delta bar X delta bar Y}
There exist $ \delta_{ 3}>0$ and $ \kappa_{ 3}\geq 1$ such that for all $ \delta\leq \delta_{ 3}$, the following is true: if for $i=1, 2$, $ \int x^{ 6} {\rm d}\mu_{i, 0}\leq \kappa_{ 0}^{ x}$, $ \int y^{ 6} {\rm d}\mu_{i, 0}\leq \kappa_{ 0}^{ y}$ (recall Lemma~\ref{lem:moments X Y}), $W(\tilde \mu_{i,0},q_0)\leq \kappa_1\gd$ and $\dist_\Pi(m_{i,0},\gamma)\leq \kappa_2 \gd$, then for all $t\in \left[0,\frac{2T}{\gd}\right]$ we have
\begin{equation}
\label{eq:bound W2b delta mu}
W(\mu_{1,t},\mu_{2,t}) \leq \kappa_3 W(\mu_{1,0},\mu_{2,0})\, .
\end{equation}
Moreover, under the same assumptions, there exist positive constants $ \kappa_{ 4}, \kappa_{ 5}, \kappa_{ 6}$ such that for any coupling $ \left\lbrace (X_{1,0},Y_{1,0}), (X_{2,0},Y_{2,0})\right\rbrace \sim \pi_{ 0}\in \mathcal{ C}( \mu_{ 1, 0}, \mu_{ 2, 0})$ of the initial condition, the solutions $(X_{1,t},Y_{1,t})$, $(X_{2,t},Y_{2,t})$ of \eqref{eq:kinetic FhN} driven by the same Brownian motion satisfy
\begin{equation}
\left(\bE\left[ \left(\Delta \tilde X_t\right)^2\right]\right)^{\frac12}
\leq e^{-(K-\kappa_4 \gd)t}\left(\bE\left[ \left(\Delta \tilde X_0\right)^2\right]\right)^{\frac12}+ \kappa_5 \gd^{1-\beta} \Lambda(\pi_{ 0})\, ,\label{eq:contract bar X}
\end{equation}
and
\begin{multline}
\label{eq:contract bar Y}
\left( \bE\left[(\Delta \tilde Y_{t})^2\right]\right)^{\frac12}
\leq e^{-\frac{b\gd}{c}t}\left(\left( \bE\left[(\Delta \tilde Y_{0})^2\right]\right)^{\frac12}+\kappa_6\gd \left( \bE\left[(\Delta \tilde X_{0})^2\right]\right)^{\frac12} \right)+ \frac{\kappa_5}{b} \gd^{1-\beta} \Lambda(\pi_{ 0}).
\end{multline}
\end{lemma}

\begin{proof}
Fix an arbitrary coupling $ \left\lbrace (X_{1,0},Y_{1,0}), (X_{2,0},Y_{2,0})\right\rbrace$ with law $ \pi_{ 0}\in \mathcal{ C}( \mu_{ 1,0}, \mu_{ 2, 0})$ and consider the solutions $(X_{i,t},Y_{i,t})$, $i=1, 2$ to \eqref{eq:kinetic FhN} with initial condition $(X_{i,0},Y_{i,0})$, $i=1, 2$ driven by the same Brownian motion. The law $ \pi_{ t}$ of $\left\lbrace (X_{1,t},Y_{1,t}), (X_{2,t},Y_{2,t})\right\rbrace$ belongs to $ \mathcal{ C}( \mu_{ 1, t}, \mu_{ 2, t})$.

\begin{description}
\item[Step 1]control on $ \Delta m_{ t}$. This part is based on the a priori control obtained in Lemmas~\ref{lem:moments X Y}, \ref{lem:close to q0} and~\ref{lem:proj on gamma}. We have
\begin{multline}
\Delta \dot m_{t}=  \delta\int_{\bbR^2}F(z+m_{1,t})\dd \tilde \mu_{1,t}(z)- \delta \int_{\bbR^2}F(z+m_{2,t})\dd \tilde \mu_{2,t}(z) \\
=  \delta DF_{\frac{\gs^2}{K}}(\gamma_{\theta_{2,0}+t})\Delta m_t + \delta h(t),
\end{multline}
where
\begin{align}
h(t):=& \left(DF_{\frac{\gs^2}{K}}(m_{2,t})- DF_{\frac{\gs^2}{K}}(\gamma_{\theta_{2,0}+t})\right) \Delta m_t\nonumber\\
&+ \int_{\bbR^2}DF(z+m_{2,t})\Delta m_t\dd \tilde \mu_{1,t}(z) - \int_{\bbR^2} DF(z+m_{2,t})\Delta m_tq_0(z)\dd z\nonumber\\
& + \int_{\bbR^2}\Big(F(z+m_{1,t})-F(z+m_{2,t})-DF(z+m_{2,t})\Delta m_t\Big)\dd \tilde \mu_{1,t}(z)\nonumber\\
& + \int_{\bbR^2} F(z+m_{2,t})\dd \tilde \mu_{1,t}(z)-\int_{\bbR^2} F(z+m_{2,t})\dd \tilde \mu_{2,t}(z), \nonumber\\
&=: h_{ 1}(t)+ \ldots + h_{ 4}(t).
\end{align}
One obtains from \eqref{eq:var_const_z} that
\begin{align}\label{eq:decomp Delta m}
\Delta m_{ t} = \Pi^{ \delta}(\theta_{2,0}+t,\theta_{2,0}) \Delta m_{ 0}+  \delta\int_{0}^t\Pi^{ \delta}(\theta_{2,0}+t,\theta_{2,0}+s)h(s)\dd s.
\end{align}
Since $m_{ 2, 0}\in E_{ 0}^{ \delta}(\kappa_2\gd)$ by assumption, adapting the proof of Lemma~\ref{lem:proj on gamma}, we can find a constant $C(T)$ such that $\left\vert m_{ 2, t}-\gamma_{\theta_{2,0}+ t} \right\vert\leq C(T) \gd$, and $ \left\vert m_{ 2, t}+\gamma_{\theta_{2,0}+t} \right\vert \leq 2R_0$ for all $t\leq \frac{2T}{\gd}$. Hence, we deduce that 
\begin{align}
\left\vert h_{ 1}(t) \right\vert &\leq \left\vert DF_{ \frac{ \sigma^{ 2}}{ K}}(m_{ 2, t})-DF_{ \frac{ \sigma^{ 2}}{ K}}(\gamma_{ \theta_{2,0}+ t}) \right\vert \left\vert \Delta m_t \right\vert= \left\vert  \left(x_{ 2, t}\right)^{ 2} - \left(\gamma_{\theta_{2,0}+ t}^{ x}\right)^{ 2} \right\vert \left\vert \Delta m_{t} \right\vert \nonumber\\
&\leq 4C(T) R_0 \gd \left\vert \Delta m_{t}\right\vert =:c_{ 1}\gd \left\vert \Delta m_{t} \right\vert.\nonumber
\end{align}
Moreover, using the notations of Lemma~\ref{lem:close to q0},
\begin{align*}
\left\vert h_{ 2}(t) \right\vert &\leq  \left\vert \int_{\bbR^2}DF(z+m_{2,t})\dd \tilde \mu_{1,t}(z) - \int_{\bbR^2} DF(z+m_{2,t}) q_0(z)\dd z \right\vert \left\vert \Delta m_{t} \right\vert\\
&=  \left\vert \int_{\bbR^2}(x+x_{2,t})^{ 2}\dd \tilde \mu_{1,t}(x, y) - \int_{\bbR^2} (x+x_{2,t})^{ 2} q_0(x, y)\dd x {\rm d}y \right\vert \left\vert \Delta  m_{t}\right\vert\\
&=\left\vert \mathbf{ E}\left[\left(\tilde X_{ 1, t}\right)^{ 2} -\left(Z_{ t}^{ x}\right)^{ 2}\right]\right\vert \left\vert \Delta m_{t}\right\vert \\
&\leq  \left( \mathbf{ E} \left[ \left(X_{ 1, t}\right)^{ 2}\right]^{ \frac{ 1}{ 2}}+ \left\vert x_{ 1, t} \right\vert + \mathbf{ E} \left[ \left(Z_{ t}^{ x}\right)^{ 2}\right]^{ \frac{ 1}{ 2}}\right)\mathbf{ E}\left[ \left\vert\tilde X_{ 1, t} -Z_{ t}^{ x}\right\vert^{ 2}\right]^{ \frac{ 1}{ 2}} \left\vert \Delta m_{ t} \right\vert\\
&\leq\left( \left( \kappa_{ 0}^{ x}\right)^{ \frac{ 1}{ 6}}+ R_0+ \frac{ \sigma}{ K^{ \frac{ 1}{ 2}}}\right) \kappa_1\gd \left\vert \Delta m_{ t} \right\vert\\
&=: c_{ 2} \delta\left\vert \Delta m_{t}\right\vert.
\end{align*}
Concerning $h_{ 3}$, a Taylor expansion shows that there a numerical constant $C$ such that $\left\vert h_{ 3}(t) \right\vert\leq C \left(\left\vert x_{ 1,t} \right\vert + \left\vert x_{ 2, t} \right\vert\right)\left\vert \Delta m_{t} \right\vert^{ 2}\leq c_{ 3} \gd \left\vert \Delta m_{t} \right\vert$, for some constant $c_{ 3}$. It remains to treat the last term $h_{ 4}$:
\begin{multline}
\vert h_4(t)\vert^2 =\left\vert \mathbf{ E} \left[\Delta \tilde X_{ t} -\left(\frac{ \left( \tilde X_{ 1, t} + x_{2, t}\right)^{ 3}}{3} -\frac{ \left(\tilde X_{ 2, t}+x_{2,t}\right)^{ 3}}{3} \right) -\Delta \tilde Y_{t}\right]\right\vert^{ 2}\\
+ \left\vert \frac{1}{c} \mathbf{ E} \left[\Delta \tilde X_{t} -b \Delta \tilde Y_{t}\right] \right\vert^{ 2},
\end{multline}
and relying on Lemma~\ref{lem:moments X Y}, Lemma~\ref{lem:close to q0} and the identity
\begin{equation}
(x_1+n)^3-(x_2+n)^3=(x_1-x_2)\big((x_1+n)^2 +(x_1+n)(x_2+n)  +(x_2+n)^2 \big),
\end{equation}
we obtain that
\begin{equation}\label{eq: bound h4}
\vert h_4(t)\vert\leq c_4\left( \left(\bE\left[\left(\Delta \tilde X_{t}\right)^2\right]\right)^\frac12+b\left(\bE\left[\left(\Delta \tilde Y_{t}\right)^2\right]\right)^{\frac12}\right),
\end{equation}
for some constant $c_{ 4}>0$. Putting everything together, we obtain, for $t\leq \frac{ 2 T}{ \delta}$,
\begin{multline}
\label{eq:Dmt_XY}
\left\vert \Delta m_t\right\vert\leq C_\Pi \left( \left \vert \Delta m_0\right\vert+c_4 \gd \int_{ 0}^{t} \left( \left(\bE\left[\left(\Delta \tilde X_{s}\right)^2\right]\right)^\frac12+b\left(\bE\left[\left(\Delta \tilde Y_{s}\right)^2\right]\right)^{\frac12}\right) {\rm d}s \right)\\
+C_\Pi (c_1+c_2+c_3)\gd^2\int_0^t \left \vert \Delta m_s\right\vert \dd s\, .
\end{multline}
With these notations at hand, for the rest of the proof, we define
\begin{equation}
\label{eq:kappa3}
\kappa_{ 3}:=C_\Pi(1+4c_4T)e^{2C_\Pi(c_1+c_2+c_3)T}+1
\end{equation}
The reason for this particular choice of $ \kappa_{ 3}$ will become clear at Step~4 below. For this choice of $ \kappa_{ 3}$, define (recall the definition of $d \left(\cdot, \cdot\right)$ in \eqref{eq:def_d})
\begin{equation}
\label{eq:t3}
t_{ 2}:=t_2(\kappa_3, \delta):=\inf\left\{t\geq 0:\, \Lambda(\pi_{ t})> \kappa_3\,  \Lambda(\pi_{ 0})\right\}.
\end{equation}
Note that $t_{ 2}$ depends on $ \kappa_{ 3}$ and $ \delta$ but also on the choice of the particular coupling $ \pi_{ 0}$ of the initial condition. For the rest of the proof, we use the shortcut :
\begin{equation}
\label{eq:dt}
\Lambda_{ t}:= \Lambda(\pi_{ t})\, .
\end{equation}

\item[Step 2] control on $ \Delta \tilde{ X}_{ t}$. We have
\begin{multline}
\frac12 \dd \left( \Delta \tilde X_{t}\right)^2=-K\left(\Delta \tilde X_{t}\right)^2\dd t\\
+\left(\Delta \tilde X_{t}\right)\Bigg(- \Delta \dot x_t
+\gd \Bigg( \Delta X_{t}-\frac{(X_{1,t})^3-(X_{2,t})^3}{3}-\Delta Y_t\Bigg)\Bigg)\dd t,
\end{multline}
so that
\begin{multline}
\frac12 \frac{\dd}{\dd t} \bE\left[\left(\Delta \tilde X_{t}\right)^2\right]\leq -(K-\gd)\bE\left[\left(\Delta \tilde X_{t}\right)^2\right]\\
+\left(\bE\left[\left(\Delta \tilde X_{t}\right)^2\right]\right)^\frac12\bigg( \left|\Delta \dot x_{t}\right|
+\gd \left|\Delta x_{t}\right| + \delta\left|\Delta y_{t} \right|+\gd \left(\bE\left[\left(\Delta \tilde Y_{t}\right)^2\right]\right)^\frac12\bigg)\\
-\frac{\gd}{3}\bE\left[\Delta \tilde X_{t}\left(\left( \tilde X_{1,t}+x_{1,t}\right)^3 -\left( \tilde X_{2,t}+x_{2,t}\right)^3\right)\right].
\end{multline}
Now, remarking that, for $x_1,x_2,n_1,n_2\in \bbR$ and $\Delta x:=x_1-x_2$, $\Delta n:=n_1-n_2$,
\begin{align}
-\Delta x &\, ((x_1+n_1)^3-(x_2+n_2)^3) \nonumber\\
&=-\Delta x\, (x_1^3-x_2^3+3(x_1^2n_1-x_2^2n_2)+3(x_1n_1^2-x_2n_2^2)+n_1^3-n_2^3)\nonumber\\
&=-\left(\Delta x\right)^2(x_1^2+x_1x_2+x_2^2+3n_1(x_1+x_2)+3n_1^2)\nonumber\\
&\qquad \qquad -\Delta x \, \Delta n\, (3x_2^2+3x_2+n_1^2+n_1n_2+n_2^2) \nonumber\\
&\leq 12 n_1^2 \left(\Delta x\right)^2-\Delta x\, \Delta n\, (3x_2^2+3x_2+n_1^2+n_1n_2+n_2^2) ,
\end{align}
we get, recalling Lemma~\ref{lem:moments X Y}, for a positive constant $C_1$,
\begin{multline}
\bE\left[\Delta \tilde X_{t}\left(\left( \tilde X_{1,t}+x_{1,t}\right)^3 -\left( \tilde X_{2,t}+x_{2,t}\right)^3\right)\right]\\ \leq C_1\left(\bE\left[\left(\Delta \tilde X_{t}\right)^2 \right]
+\left|\Delta x_{t}\right|\left( \bE\left[\left(\Delta \tilde X_{t}\right)^2 \right] \right)^\frac12\right).
\end{multline}
Proceeding similarly, we obtain, for some constant $C_2$,
\begin{multline}
\frac{1}{\gd}\left| \Delta \dot x_{t}\right|=\left|\bE\left[\Delta X_{t}-\frac{\left(X_{1,t}\right)^3-\left(X_{2,t}\right)^3}{3}-\Delta Y_{t} \right]\right|\\
\leq C_2\left( \left( \bE\left[\left(\Delta \tilde X_{t}\right)^2 \right] \right)^\frac12+\left( \bE\left[\left(\Delta \tilde Y_{t}\right)^2 \right] \right)^\frac12+\left|\Delta m_t\right|\right).
\end{multline}
Gathering all these estimates, we get, for positive constants $\kappa_4$ and $C_3$, for $t\leq t_2$,
\begin{align}
\frac12 \frac{\dd}{\dd t} \bE\left[\left(\Delta \tilde X_{t}\right)^2\right]&\leq -(K-\kappa_4\gd)\bE\left[\left(\Delta \tilde X_{t}\right)^2\right] \nonumber\\
&+C_3\gd \left(\bE\left[\left(\Delta \tilde X_{t}\right)^2\right]\right)^\frac12\left(\left( \bE\left[\left(\Delta \tilde Y_{t}\right)^2 \right] \right)^\frac12+\left|\Delta m_{t}\right| \right) \nonumber\\
&\leq -(K-\kappa_4\gd)\bE\left[\left(\Delta \tilde X_{t}\right)^2\right]+C_3 \kappa_{ 3}\gd \left(b^{ -1} + \delta^{ - \beta}\right) \Lambda_{ 0}\left(\bE\left[\left(\Delta \tilde X_{t}\right)^2\right]\right)^\frac12\, , \label{eq:bound EDX}
\end{align}
where we recall the definition of $ \Lambda_{ t}$ in \eqref{eq:dt} and the definition of $t_{ 2}$ in \eqref{eq:t3}. By Lemma ~\ref{lem:gronwall_sqrt}, this implies that for all $t\leq t_2$,
\begin{equation}
 \left(\bE\left[\left(\Delta \tilde X_{t}\right)^2\right]\right)^{\frac12}\leq \max\left\{ \left(\bE\left[\left(\Delta \tilde X_{0}\right)^2\right]\right)^{\frac12}, \frac{ C_3 \kappa_3 (b^{ -1} + \delta^{ - \beta}) }{K-\kappa_4 \delta}\gd \Lambda_{ 0}\right\}.
\end{equation}
So taking $\gd$ small enough (depending on $\kappa_3$), we deduce in particular that, since $ \beta\in (0, 1)$ by hypothesis, for $t\leq t_2$
\begin{equation}
\label{eq:EDX_d0}
\bE\left[\left(\Delta \tilde X_{t}\right)^2\right]^{ \frac{ 1}{ 2}}\leq \Lambda_{ 0}.
\end{equation}
\item[Step 3] control on $ \Delta \tilde{ Y}_{ t}$. A simple calculation leads to
\begin{equation}
\frac{c}{2\gd}\dd \left(\Delta \tilde Y_{t} \right)^2= \left(\Delta \tilde X_{t} \, \Delta \tilde Y_{t} -b \left(\Delta \tilde Y_{t} \right)^2\right)\dd t,
\end{equation}
so that, for $t\leq t_2$, using \eqref{eq:EDX_d0},
\begin{equation}\label{eq:bound EDY}
\frac{c}{2\gd}\frac{\dd}{\dd t} \bE\left[\left(\Delta \tilde Y_{t}\right)^2\right]\leq -b \bE\left[\left(\Delta \tilde Y_{t}\right)^2\right]+ \Lambda_{ 0} \left(\bE\left[\left(\Delta \tilde Y_{t}\right)^2\right]\right)^\frac12,
\end{equation}
which implies in particular, by Lemma~\ref{lem:gronwall_sqrt}, that for all $t\leq t_2$,
\begin{equation}
\label{eq:EDY_d0}
\left( \bE\left[(\Delta \tilde Y_{t})^2\right]\right)^{\frac12}\leq \max\left\{\left(\bbE\left[\left(\Delta \tilde Y_0\right)^2\right]\right)^{\frac12},\frac{ \Lambda_{ 0}}{b}\right\}\leq \frac{ \Lambda_{ 0}}{b}.
\end{equation}
\item[Step 4] proof of \eqref{eq:bound W2b delta mu}. Let us first prove that, for the choice of $ \kappa_{ 3}$ in \eqref{eq:kappa3}, we have, for $ \delta\leq \delta_{ 3}$ for some $ \delta_{ 3}$ sufficiently small
\begin{equation}
\label{eq:t2_VS_Tgam}
t_{ 2}(\kappa_{ 3}, \delta) \geq \frac{ 2T}{ \delta}.
\end{equation}
Suppose that $ \delta_{ 3}$ is chosen sufficiently small so that \eqref{eq:EDX_d0} is true and such that $ \delta_{ 3}\leq 1$. Suppose that \eqref{eq:t2_VS_Tgam} does not hold for some $ \delta\leq \delta_{ 3}$: $t_{ 2} < \frac{ 2T}{ \delta}$. For this $ \delta$, for $t\leq t_{ 2}< \frac{ 2T}{ \delta}$, incorporating \eqref{eq:EDX_d0} and \eqref{eq:EDY_d0} into \eqref{eq:Dmt_XY} and applying Gr\"onwall's Lemma gives, for $t\leq t_{ 2}< \frac{ 2 T_{ \gamma}}{ \delta}$,
\begin{align}
  \delta^{ \beta}\left\vert \Delta m_{ t} \right\vert &\leq  C_\Pi \left( \delta^{ \beta}\left \vert \Delta m_0\right\vert+2c_4 \gd^{ 1+ \beta} t\, \Lambda_{ 0}\right)  e^{C_\Pi(c_1+c_2+c_3) \delta^2  t}\, ,\nonumber\\
&\leq \Lambda_{ 0} C_\Pi \left( 1+4c_4T\right)  e^{2C_\Pi(c_1+c_2+c_3)T}= \left( \kappa_{ 3}-1\right) \Lambda_{ 0},\label{eq:bound delta mt}
\end{align}
where we have used $\gd_{ 3}\leq 1$. Collecting \eqref{eq:EDX_d0}, \eqref{eq:EDY_d0} and \eqref{eq:bound delta mt}, we obtain that for all $t\leq t_{ 2}$, $ \Lambda_{ t}< \kappa_{ 3} \Lambda_{ 0}$. By continuity, there exists some $ \epsilon>0$ such that $ \Lambda_{ t}< \kappa_{ 3} \Lambda_{ 0}$ for $t\in [t_{ 2}, t_{ 2}+ \epsilon)$, which contradicts the definition of $t_{ 2}$. Hence, \eqref{eq:t2_VS_Tgam} follows. In particular, since $W$ is an infimum over all possible coupling, for $t\leq \frac{ 2T}{ \delta}$, $W( \mu_{ 1, t}, \mu_{ 2, t}) \leq \kappa_{ 3} \Lambda_{ 0}= \kappa_{ 3} \Lambda(\pi_{ 0})$. Since this is true for all possible coupling of the initial condition, \eqref{eq:bound W2b delta mu} follows.
\item[Step 5]
Let us now prove \eqref{eq:contract bar X}: from \eqref{eq:bound EDX} and applying Lemma~\ref{lem:gronwall_sqrt}, we deduce that, for all $t\leq\frac{2T}{\gd}$,
\begin{equation}\label{eq: bound EDX2}
\left(\bE\left[\left(\Delta \tilde X_{t}\right)^2\right]\right)^\frac12\leq e^{-(K-\kappa_4\gd)t}\left(\bE\left[\left(\Delta \tilde X_{0}\right)^2\right]\right)^\frac12+\frac{ C_3\kappa_3\left(b^{ -1}+ \delta^{ - \beta}\right)}{K-\kappa_4\gd} \gd\, \Lambda_{ 0},
\end{equation}
which implies \eqref{eq:contract bar X} with $\kappa_5=\frac{2C_3\kappa_3b\left(1+b\right)}{K}$ as soon as $\gd\leq \frac{K}{2\kappa_4}$. Let us now turn to the proof of \eqref{eq:contract bar Y}: from \eqref{eq:bound EDY}, we have, for all $t\leq\frac{2T}{\gd}$,
\begin{multline}
\frac{c}{2\gd}\frac{\dd}{\dd t} \bE\left[\left(\Delta \tilde Y_{t}\right)^2\right]\leq -b \bE\left[\left(\Delta \tilde Y_{t}\right)^2\right]\\+\left(e^{-(K-\kappa_4\gd)t}\left(\bE\left[\left(\Delta \tilde X_{0}\right)^2\right]\right)^\frac12+\kappa_5\gd^{1-\beta} \, \Lambda_{ 0}\right) \left(\bE\left[\left(\Delta \tilde Y_{t}\right)^2\right]\right)^\frac12,
\end{multline}
so that, by Lemma~\ref{lem:gronwall_sqrt} again implies
\begin{multline}
\left( \bE\left[(\Delta \tilde Y_{t})^2\right]\right)^{\frac12}\leq e^{-\frac{b\gd}{c}t}\left(\left( \bE\left[(\Delta \tilde Y_{0})^2\right]\right)^{\frac12}+\frac{\gd}{cK-(c\kappa_4+b)\gd}\left( \bE\left[(\Delta \tilde X_{0})^2\right]\right)^{\frac12} \right)\\
+\kappa_5 \gd^{1-\beta} \, \frac{ \Lambda_{ 0}}{b},
\end{multline}
which leads to \eqref{eq:contract bar Y} with $\kappa_6=\frac{2}{cK}$, as soon as $\gd\leq \frac{cK}{2(c\kappa_4+b)}$.
\end{description}
This concludes the proof of Lemma~\ref{lem:bound delta bar X delta bar Y}.
\end{proof}

\begin{lemma}\label{lem:encadre}
There exists a $\gd_4>0$ and positive constants $\kappa_7,\kappa_8,\kappa_9$ such that if $\gd\leq \gd_4$, if the hypotheses of Lemma~\ref{lem:bound delta bar X delta bar Y} are satisfied, and if moreover we have the initial bounds $W(\tilde \mu_{1,0},\tilde \mu_{2,0})\leq \kappa_7 \gd^{ 2}\vert \Delta \theta_0\vert$ and $\left\vert p^\perp_{\theta_{1,0}}(m_{1,0}-\gamma_{\theta_{1,0}})-p^\perp_{\theta_{2,0}}(m_{2,0}-\gamma_{\theta_{2,0}})\right\vert\leq \kappa_8\gd^\frac32 \vert \Delta \theta_0\vert$, then for all $t\in \left[0,\frac{2T}{\gd}\right]$ we have
\begin{equation}\label{eq:encadre delta m}
\left\vert \vert \Delta \theta_t\vert-\vert \Delta \theta_0\vert\right\vert\leq  \kappa_9 \gd^\frac{1}{2}\vert \Delta \theta_0\vert,
\end{equation}
and for all $t\in \left[\frac{T}{\gd},\frac{2T}{\gd}\right]$ we have
\begin{equation}\label{eq:contract W2b Delta tilde mu}
W(\tilde \mu_{1,t},\tilde \mu_{2,t})\leq \kappa_7 \gd^{ 2} \vert \Delta \theta_t\vert,
\end{equation}
and
\begin{equation}\label{eq:contract Delta p perp}
\left\vert p^\perp_{\theta_{1,t}}(m_{1,t}-\gamma_{\theta_{1,t}})-p^\perp_{\theta_{2,t}}(m_{2,t}-\gamma_{\theta_{2,t}})\right\vert\leq \kappa_8\gd^\frac32 \vert \Delta \theta_t\vert.
\end{equation}
\end{lemma}

\begin{proof}
Define $\kappa_7=16 \kappa_{ 5} \left\vert \dot \gamma \right\vert $, fix a $\kappa_8>0 $ whose value will be chosen later, and suppose that $W(\tilde \mu_{1,0},\tilde \mu_{2,0})\leq \kappa_7 \gd^{ 2}\vert \Delta \theta_0\vert$ and $\left\vert p^\perp_{\theta_{1,0}}(m_{1,0}-\gamma_{\theta_{1,0}})-p^\perp_{\theta_{2,0}}(m_{2,0}-\gamma_{\theta_{2,0}})\right\vert\leq \kappa_8\gd^\frac32 \vert \Delta \theta_0\vert$. 

In what follows, we also fix some $ \varepsilon>0$ and consider some coupling $ \left\lbrace \left( X_{ 1, 0}, Y_{ 1, 0}\right), \left( X_{ 2, 0}, Y_{ 2, 0}\right)\right\rbrace \sim \pi_{ 0}\in \mathcal{ C}(\mu_{ 0, 1}, \mu_{ 2, 0})$ of the initial condition such that 
\begin{equation}
\label{eq:coupling_mu12}
\Lambda(\pi_{ 0}) < W(\mu_{ 1,0},\mu_{ 2, 0}) + \varepsilon.
\end{equation}

Remark first that, since $m_{1,0}-m_{2,0}=m_{1,0}-\gamma_{\theta_{1,0}}+\gamma_{\theta_{1,0}}-\gamma_{\theta_{2,0}}+\gamma_{\theta_{2,0}}-m_{2,0}$ and $m_{i,0}-\gamma_{\theta_{i,0}}=p^\perp_{\theta_{i,0}}(m_{i,0}-\gamma_{\theta_{i,0}})e_{\theta_{i,0}}$ for $i=1,2$, we have
\begin{multline}
\vert \Delta m_0\vert \leq \left\vert \gamma_{\theta_{1,0}}-\gamma_{\theta_{2,0}}\right\vert +\left\vert \left(p^\perp_{\theta_{1,0}}(m_{1,0}-\gamma_{\theta_{1,0}})-p^\perp_{\theta_{2,0}}(m_{2,0}-\gamma_{\theta_{2,0}}) \right)e_{\theta_{2,0}} \right\vert\\
+\left\vert p_{\theta_{1,0}}^\perp(m_{1,0}-\gamma_{\theta_{1,0}})(e_{\theta_{1,0}}-e_{\theta_{2,0}})\right\vert,
\end{multline}
so that
\begin{equation}\label{eq:bound Delta m with Delta p perp}
\vert \Delta m_0\vert \leq \Big( \vert \dot \gamma\vert_\infty +C_\Pi\left(\kappa_8 \gd^\frac12+C_\Pi\kappa_2 \gd)\right) \Big) \delta\vert \Delta \theta_0\vert\leq 2 \delta\vert \dot \gamma\vert \vert \Delta \theta_0\vert,
\end{equation}
for $\gd$ small enough (depending on $\kappa_8$). By definition of $ \kappa_{ 7}$, this means in particular that
\begin{equation}\label{eq:bound W Delta theta}
W(\mu_{1,0},\mu_{2,0})=\max \left( \delta^{ \beta} \left\vert \Delta m_{ 0} \right\vert, W(\tilde{ \mu}_{1,0}, \tilde{ \mu}_{2,0})\right)\leq 2\vert \dot \gamma\vert \gd^{1+\beta} \vert \Delta \theta_0\vert,
\end{equation}
provided that $\gd\leq \left( \frac{2 \left\vert \dot \gamma \right\vert}{ \kappa_{ 7}}\right)^{ \frac{ 1}{ 1- \beta}}$. Now remark that
\begin{equation}
\left \vert \left\vert \Delta \theta_t \right\vert - \left\vert \Delta \theta_0 \right\vert \right\vert
\leq \left \vert \theta_{1,t}-(\theta_{1,0}+ t)\right\vert+\left \vert \theta_{2,t}-(\theta_{2,0}+ t)\right\vert,
\end{equation}
and relying on similar arguments as in Lemma~\ref{lem:proj on gamma}, we have $ \vert \theta_{i,t}-(\theta_{i,0}+ t)\vert \leq C_4 \gd t$, for $t\in \left[0,\frac{2T}{\gd}\right]$, $i=1,2$, and some positive constant $C_4$.
So, for $\vert \Delta \theta_0\vert \geq \gd^{-\frac{1}{2}}$, we have $ \left \vert \left\vert \Delta \theta_t \right\vert - \left\vert \Delta \theta_0 \right\vert \right\vert \leq 2 C_4 T \gd^\frac{1}{2}\left\vert \Delta \theta_0 \right\vert$. Suppose now that $\vert \Delta \theta_0\vert \leq \gd^{-\frac{1}{2}}$. From Lemma~\ref{lem:proj} and Lemma~\ref{lem:proj on gamma} we deduce that
\begin{equation}
\label{eq:Delta_theta_t}
\Delta \theta_t = \frac{ 1}{ \delta} \left(p_{\theta_{2,t}}(\Delta m_t)+O(\gd \vert \Delta m_t\vert+\vert \Delta m_t\vert^2)\right)= \frac{ 1}{ \delta} \left(p_{\theta_{2,0}+t}(\Delta m_t)+O(\gd \vert \Delta m_t\vert+\vert \Delta m_t\vert^2)\right).
\end{equation}



The first point is to be able to replace the term $p_{\theta_{2,0}+t}(\Delta m_t)$ by $p_{\theta_{2,0}}(\Delta m_0)$, namely to prove that there exists some constant $C_{ 5}^{ \prime}$, independent of $ \varepsilon$ such that for all $t\in \left[0, \frac{ 2T}{ \delta}\right]$
\begin{align}
\label{eq:p_Delta_Delta0}
\big\vert p_{\theta_{2,0}+ t}(\Delta m_t)-&p_{\theta_{2,0}}(\Delta m_0) \big\vert \leq C_{ 5}^{ \prime} \left(\gd^{ 2} \vert \Delta \theta_0\vert + \varepsilon\right)\, .
\end{align}
Indeed, noting first that $p_{\theta_{2,0}+t}(\Pi^{ \delta}(\theta_{2,0}+t,\theta_{2,0})\Delta m_0)=p_{\theta_{2,0}}(\Delta m_0)$, we obtain from \eqref{eq:Dmt_XY} and \eqref{eq:bound W2b delta mu} that for $t\in\left[0,\frac{2T}{\gd}\right]$
\begin{align}
\big\vert p_{\theta_{2,0}+ t}(\Delta m_t)-p_{\theta_{2,0}}(\Delta m_0) \big\vert &= \big\vert p_{\theta_{2,0}+ t}(\Delta m_t)-p_{\theta_{2,0}+t}(\Pi^{ \delta}(\theta_{2,0}+t,\theta_{2,0})\Delta m_0) \big\vert \nonumber \\
&\leq C_\Pi (c_1+c_2+c_3) \kappa_3  \gd^{2-\beta} t\,  W(\mu_{1,0},\mu_{2,0})\nonumber\\
&+C_\Pi \gd \int_0^t c_4\left( \left(\bE\left[\left(\Delta \tilde X_{s}\right)^2\right]\right)^\frac12+b\left(\bE\left[\left(\Delta \tilde Y_{s}\right)^2\right]\right)^{\frac12}\right)\dd s\, . \label{aux:deltap}
\end{align}
By definition of the coupling $ \pi_{ 0}$ in \eqref{eq:coupling_mu12} and the assumption $W(\tilde \mu_{1,0},\tilde \mu_{2,0})\leq \kappa_7 \gd^{ 2}\vert \Delta \theta_0\vert$, we have $\left(\bE\left[\left(\Delta \tilde X_{0}\right)^2\right]\right)^{\frac12} \leq \kappa_7\gd^{ 2} \vert \Delta \theta_0\vert + \varepsilon$ and $b\left(\bE\left[\left(\Delta \tilde Y_{0}\right)^2\right]\right)^{\frac12}\leq \kappa_7 \gd^{ 2}  \vert \Delta \theta_0\vert + \varepsilon$. Hence, using \eqref{eq:contract bar X}, \eqref{eq:contract bar Y},  the integral term in \eqref{aux:deltap} can be bounded above by
\begin{align*}
C_{ 5} \gd t \left(\gd^{ 2} \vert \Delta \theta_0\vert  + \varepsilon\right)
\end{align*}
for some constant $C_{ 5}>0$, independent of $ \varepsilon$. Thus, \eqref{eq:p_Delta_Delta0} is a direct consequence of \eqref{eq:bound W Delta theta}.

Secondly, remark that
\begin{equation}
\Delta \theta_0 = \frac{ 1}{ \delta} \left(p_{\theta_{2,0}}(\Delta m_0) +O(\gd \vert \Delta m_0\vert+\vert \Delta m_0\vert^2)\right),
\end{equation}
and that \eqref{eq:bound delta mt} implies that $\vert \Delta m_t\vert \leq C_6\vert \Delta m_0\vert$ for some constant $C_6$. Gathering these estimates, \eqref{eq:Delta_theta_t} and \eqref{eq:p_Delta_Delta0}, we obtain (recall that $\vert \Delta m_0 \vert \leq 2 \gd \vert \dot \gamma\vert \vert \Delta \theta_0\vert$ and $\vert \Delta \theta_0\vert \leq \gd^{-\frac{1}{2}}$):
\begin{equation}
\Delta \theta_t - \Delta \theta_0= O\left(\gd^\frac{1}{2} \vert \Delta \theta_0\vert + \frac{ \varepsilon}{ \delta}\right).
\end{equation}
Since obviously neither $ \Delta \theta_{ t}$ nor $ \Delta \theta_{ 0}$ depend on $\varepsilon$, one can make $ \varepsilon \searrow 0$ in the previous estimate and obtain \eqref{eq:encadre delta m} with a $\kappa_9$ that does not depend on $\kappa_8$, for $\gd$ small enough (depending on $\kappa_8$). 

It remains to prove \eqref{eq:contract W2b Delta tilde mu} and \eqref{eq:contract Delta p perp}. Using \eqref{eq:coupling_mu12}, \eqref{eq:bound W Delta theta} and the assumption on the initial condition $X_{ 0}$ into \eqref{eq:contract bar X}, we obtain
\begin{align*}
\left(\bE\left[ \left(\Delta \tilde X_t\right)^2\right]\right)^{\frac12}
\leq \left(e^{-(K-\kappa_4 \gd)t} \kappa_7 + 2\kappa_5 \vert \dot \gamma\vert \right)\gd^{2} \vert \Delta \theta_0\vert+ \varepsilon \left(e^{-(K-\kappa_4 \gd)t} + \kappa_5 \gd^{1-\beta}\right)\, .
\end{align*}
Using \eqref{eq:encadre delta m}, we have
\begin{equation}
\left\vert \Delta\theta_{ 0} \right\vert \leq \frac{ \left\vert \Delta\theta_{ t} \right\vert}{ 1- \kappa_{ 9} \delta^{ \frac{ 1- \beta}{ 2}}}
\end{equation}
which gives
\begin{equation}
\left(\bE\left[ \left(\Delta \tilde X_t\right)^2\right]\right)^{\frac12}
\leq  \left(e^{-(K-\kappa_4 \gd)t} + \frac{ 2\kappa_5}{ \kappa_{ 7}} \vert \dot \gamma\vert \right) \kappa_{ 7}\delta^{ 2}  \frac{ \left\vert \Delta\theta_{ t} \right\vert}{ 1- \kappa_{ 9} \delta^{ \frac{ 1- \beta}{ 2}}}  + \varepsilon \left(e^{-(K-\kappa_4 \gd)t} + \kappa_5 \gd^{1-\beta}\right)\, .
\end{equation}
Choosing $ \delta\leq \frac{ 1}{ 2 \kappa_{ 9}}$, we have
\begin{equation}
\left(\bE\left[ \left(\Delta \tilde X_t\right)^2\right]\right)^{\frac12}
\leq  \left(e^{-(K-\kappa_4 \gd)t} + \frac{ 2\kappa_5 \vert \dot \gamma\vert}{ \kappa_{ 7}} \right) 2\kappa_{ 7}\delta^{ 2}  \left\vert \Delta\theta_{ t} \right\vert + \varepsilon \left(e^{-(K-\kappa_4 \gd)t} + \kappa_5 \gd^{1-\beta}\right)\, .
\end{equation}
Recalling that $ \kappa_{ 7}= 16 \kappa_{ 5} \left\vert \dot \gamma \right\vert$ and choosing $ \delta\leq \frac{ K}{ 2 \kappa_{ 4}}$ and $\gd\leq \frac{KT}{6\ln 2}$ (so that $e^{ - \frac{ K}{ 2} \frac{T}{2}} \leq \frac{ 1}{ 8}$), we obtain for $t\in \left[ \frac{ T}{ \delta}, \frac{ 2T}{ \delta}\right]$,
\begin{equation}
\left(\bE\left[ \left(\Delta \tilde X_t\right)^2\right]\right)^{\frac12}
\leq \frac{ \kappa_{ 7}}{ 2}\delta^{ 2}  \left\vert \Delta\theta_{ t} \right\vert + \varepsilon \left(e^{-(K-\kappa_4 \gd)t} + \kappa_5 \gd^{1-\beta}\right)\, .
\end{equation}
Concerning \eqref{eq:contract bar Y}, by similar arguments, we have
\begin{multline}
b\left( \bE\left[(\Delta \tilde Y_{t})^2\right]\right)^{\frac12}
\leq \left[e^{-\frac{b\gd}{c}t}\left(1 + b \kappa_6\gd\right) + \frac{ 2 \kappa_{ 5} \left\vert \dot \gamma \right\vert}{ \kappa_{ 7}}\right]2 \kappa_7 \gd^{ 2}\vert \Delta \theta_t\vert\\ + \varepsilon \left(\kappa_5 \gd^{1-\beta}  + e^{-\frac{b\gd}{c}t} \left(1+ b \kappa_6\gd\right)\right)\, ,
\end{multline}
which gives, by definition of $ \kappa_{ 7}$ and $T$ (recall \eqref{hyp T}) and choosing $ \delta\leq \frac{ 1}{ b \kappa_{ 6}}$, for $t\in \left[ \frac{ T}{ \delta}, \frac{ 2T}{ \delta}\right]$:
\begin{equation}
b\left( \bE\left[(\Delta \tilde Y_{t})^2\right]\right)^{\frac12}
\leq \frac{ \kappa_7}{ 2} \gd^{ 2}\vert \Delta \theta_t\vert + \varepsilon \left(\kappa_5 \gd^{1-\beta}  + e^{-\frac{b\gd}{c}t} \left(1+ b \kappa_6\gd\right)\right)\, .
\end{equation}
Similar estimates are \emph{a fortiori} valid for the infimum $ W \left( \tilde{ \mu}_{ 1, t}, \tilde{ \mu}_{ 2, t}\right)$ and letting $ \varepsilon\searrow 0$, we obtain the result \eqref{eq:contract W2b Delta tilde mu}.

Now, recalling Lemma~\ref{lem:proj on gamma},
\begin{equation}
\left\vert p^\perp_{\theta_{1,t}}(m_{1,t}-\gamma_{\theta_{1,t}})-p^\perp_{\theta_{2,t}}(m_{2,t}-\gamma_{\theta_{2,t}})\right\vert \leq 2 C_\Pi \kappa_2 \gd,
\end{equation}
so that \eqref{eq:contract Delta p perp} is directly valid for $\vert \Delta \theta_t\vert\geq 2 C_\Pi \kappa_2 \gd^{ \frac32}$. Suppose now that $\vert \Delta \theta_t\vert\leq 2 C_\Pi \kappa_2 \gd^{ \frac32}$.
Relying again on Lemma~\ref{lem:proj on gamma}, and using the fact that $p ^\perp_{\theta_{2,t}}(\dot \gamma_{\theta_{2,t}})=0$, we have for $t\in\left[0,\frac{2T}{\gd}\right]$,
\begin{equation}
p^\perp_{\theta_{1,t}}(m_{1,t}-\gamma_{\theta_{1,t}})-p^\perp_{\theta_{2,t}}(m_{2,t}-\gamma_{\theta_{2,t}})
=p^\perp_{\theta_{2,t}}(\Delta m_t)+\gd^{ 2} O(\vert \Delta \theta_t\vert)+ \delta^{ 2}O(\vert \Delta \theta_t\vert^2),
\end{equation}
and since $\vert \theta_{2,t}-(\theta_{2,0}+ t)\vert \leq C_4 \gd t$ and
\begin{equation}
\vert \Delta m_t\vert\leq \gd^{-\beta} W(\mu_{1,t},\mu_{2,t})\leq \gd^{-\beta}\kappa_3 W(\mu_{1,0},\mu_{2,0})\leq 2 \delta\vert \dot\gamma\vert  \vert \Delta \theta_0\vert,
\end{equation} 
we obtain
\begin{equation}
p^\perp_{\theta_{1,t}}(m_{1,t}-\gamma_{\theta_{1,t}})-p^\perp_{\theta_{2,t}}(m_{2,t}-\gamma_{\theta_{2,t}})
=p^\perp_{\theta_{2,0}+ t}(\Delta m_t)+O\left(\gd^\frac32 \vert \Delta \theta_t\vert\right).
\end{equation}
Similarly,
\begin{equation}
p^\perp_{\theta_{1,0}}(m_{1,0}-\gamma_{\theta_{1,0}})-p^\perp_{\theta_{2,0}}(m_{2,0}-\gamma_{\theta_{2,0}})
=p^\perp_{\theta_{2,0}}(\Delta m_0)+O\left(\gd^\frac32 \vert \Delta \theta_t\vert\right),
\end{equation}
and using the decomposition \eqref{eq:decomp Delta m} and relying on similar estimates as made above,
\begin{equation}
p^\perp_{\theta_{2,0}+ t}(\Delta m_t) = e^{-\gl \gd t}p^\perp_{\theta_{2,0}}(\Delta m_0)+O\left(\gd^\frac32 \vert \Delta \theta_t\vert\right).
\end{equation}
So, using again \eqref{eq:encadre delta m}, there exists a constant $C_7$ that does not depend on $\kappa_8$ such that
\begin{equation}
\left\vert p^\perp_{\theta_{1,t}}(m_{1,t}-\gamma_{\theta_{1,t}})-p^\perp_{\theta_{2,t}}(m_{2,t}-\gamma_{\theta_{2,t}})\right\vert \leq \left(2 e^{-\gl \gd t}  +\frac{C_7}{\kappa_8}\right) \kappa_8\gd^\frac32\vert\Delta \theta_t\vert,
\end{equation}
which implies \eqref{eq:contract Delta p perp}, taking $\kappa_8\geq 2 C_7$, and recalling \eqref{hyp T}.
\end{proof}

\section{Fixed point}
\label{sec:fixed_point}
Recall the definitions of the constants $ \kappa_{ i}$, $i=0, \ldots, 8$ appearing in Section~\ref{sec:closeness} and Section~\ref{sec:contract}. For $\gd\leq \gd_4$ (recall Lemma~\ref{lem:encadre}), consider the space $\cF=\cF(\gd)$ composed of the functions $f: \mathbb{ S}_{ \delta}\rightarrow \cP_2$ that satisfy
\begin{enumerate}
\item 
\begin{equation}\label{def f moments}
\int_{\bbR^2}x^6 f(\theta)( {\rm d}z)\leq \kappa^x_0 \qquad \text{and} \qquad \int_{\bbR^2} y^6 f(\theta)( {\rm d} z)  \leq \kappa_0^y, \qquad \text{for all } \theta  \in \mathbb{ S}_{ \delta},
\end{equation}
\item
\begin{equation}\label{def f close q0}
W(\tilde{f(\theta)},q_0)\leq \kappa_1\gd,  \quad  \text{for all } \theta  \in \mathbb{ S}_{ \delta},
\end{equation}
\item
\begin{equation}\label{def f close gamma}
\int_{\bbR^2} z f(\theta)( {\rm d}z) \in E_\theta(\kappa_2 \gd), \quad  \text{for all } \theta  \in \mathbb{ S}_{ \delta},
\end{equation}
\item 
\begin{equation}\label{def f W_2b lipsch}
W(\tilde{f(\theta_1)},\tilde {f(\theta_2)})\leq \kappa_7 \gd^{ 2} \vert \theta_1-\theta_2\vert ,  \quad  \text{for all } \theta_1,\theta_2  \in \mathbb{ S}_{ \delta},
\end{equation}
\item
\begin{multline}
\left\vert  p^\perp_{\theta_1}\left(\int_{\bbR ^2} z f(\theta_1)({\rm d}z)-\gamma_{\theta_1}\right) - p^\perp_{\theta_2}\left(\int_{\bbR ^2} z f(\theta_2)({\rm d}z)-\gamma_{\theta_2}\right)  \right\vert \leq \kappa_8 \gd^\frac32 \vert \theta_1-\theta_2\vert, \\
  \text{for all } \theta_1,\theta_2  \in \mathbb{ S}_{ \delta}.
\end{multline}
\end{enumerate}
We define, for $\theta\in \mathbb{ S}_{ \delta}$, the distance $W_\theta=W_\theta(\gd,b,\beta)$ on probability measures $\nu$ that satisfy $\int_{\bbR^2} z\nu( {\rm d}z)\in E_\theta(\kappa_2\gd)$ as follows: 
\begin{align}
W_\theta(\nu_1,\nu_2):=\max\Bigg\{&\gd^\beta\left|p^\perp_\theta\left(\int_{\bbR^2}z\dd \nu_1(z)-\int_{\bbR^2}z\dd \nu_2(z)\right) \right|, W \left(\tilde{ \nu}_{ 1}, \tilde{ \nu}_{ 2}\right)\Bigg\},
\end{align}
where $W$ is defined in \eqref{eq:def_W_1}. We consider the distance $\dist_\cF$ on $\cF$ defined as
\begin{equation}
\dist_\cF(f_1,f_2)=\sup_{\theta\in \mathbb{ S}_{ \delta}} W_\theta(f_1(\theta),f_2(\theta)).
\end{equation}
Remark that, on $ \mathcal{ F}$, $\frac{1}{C_\Pi}W\leq W_\theta\leq C_\Pi W$ and this distance is thus equivalent to the distance $\sup_{\theta\in \mathbb{ S}_{ \delta}}W_2(f_1(\theta),f_2(\theta))$. $\cF$ is then a closed subspace of $C( \mathbb{ S}_{ \delta},\cP_2)$ (where $\cP_2$ is endowed with the standard Wasserstein-2 distance), and is thus complete when endowed with $\dist_\cF$.

\medskip 

For $t\in\left[0,\frac{2T}{\gd}\right]$ we consider the map $g_{t,f}: \mathbb{ S}_{ \delta}\rightarrow \mathbb{ S}_{ \delta}$ as
\begin{equation}
g_{t,f}(\theta)=\proj\left( \int_{\bbR^2}z \mu_t({\rm d}z)\right),
\end{equation}
where $\mu_t$ is the solution of \eqref{eq:PDER kinetic FhN} with $\mu_0=f(\theta)$. Lemma \ref{lem:encadre} shows that, for $\gd$ small enough, $g_{t,f}$ is a bijection for any $f\in\cF$.
We consider the map $\Phi_t:\cF\rightarrow C( \mathbb{ S}_{ \delta},\cP_2)$ defined as
\begin{equation}
\Phi_t(\theta)=\mu_{ \frac{ t}{ \delta}},
\end{equation}
where $\mu$ is the solution of \eqref{eq:PDER kinetic FhN} with $\mu_0=f\left(g^{-1}_{t,f}(\theta)\right)$.

\begin{lemma}\label{lem:lipsch}
If $\gd\leq \gd_4$ and $f\in \cF$, then $\Phi_t(f)\in \cF$ for all $t\in \left[T,2T\right]$.
\end{lemma}

\begin{proof}
This is a direct consequence of Lemma~\ref{lem:moments X Y}, Lemma~\ref{lem:close to q0}, Lemma~\ref{lem:proj on gamma}, Lemma~\ref{lem:bound delta bar X delta bar Y} and Lemma~\ref{lem:encadre}.
\end{proof}

\begin{lemma}\label{lem:X contracts}
There exist $\gd_5>0$ and $\kappa_{10}>0$ such that if $\gd\leq \gd_5$ and $f_1,f_2,\in \cF$, then for all $t\in[T,2T]$ we have
\begin{equation}
\dist_\cF(\Phi_t(f_1),\Phi_t(f_2))\leq  \left(\max\left\{e^{-\gl t}, e^{-\frac{b}{c}t} \right\} + \kappa_{10} \left(\gd^\beta+\gd^{1-\beta}\right) \right)\dist_\cF(f_1,f_2).
\end{equation} 
\end{lemma}

\begin{proof}
Fix a $t\in [T,2T]$. For $\theta\in \mathbb{ S}_{ \delta}$, consider solutions $\mu_{i,s}$ to \eqref{eq:PDER kinetic FhN} starting from $\mu_{i,0}=f_i(\theta_{i,0})$ and such that $\theta_{i,\frac{t}{\gd}}=\theta$, $i=1,2$. Our aim is to bound $W_\theta\left(\mu_{1,\frac{t}{\gd}},\mu_{2,\frac{t}{\gd}}\right)$.
To do this we consider also the solution $\mu_{3,s}$ starting from $f_1(\theta_{2,0})$.
Remark first that
\begin{equation}
W_\theta\left(\tilde \mu_{1,\frac{t}{\gd}},\tilde \mu_{2,\frac{t}{\gd}}\right)=W\left(\tilde \mu_{1,\frac{t}{\gd}},\tilde \mu_{2,\frac{t}{\gd}}\right)\leq W\left(\tilde \mu_{1,\frac{t}{\gd}},\tilde \mu_{3,\frac{t}{\gd}}\right)+W\left(\tilde \mu_{3,\frac{t}{\gd}},\tilde \mu_{2,\frac{t}{\gd}}\right).
\end{equation}
On one hand,  
\begin{align}
W\left(\tilde \mu_{1,\frac{t}{\gd}},\tilde \mu_{3,\frac{t}{\gd}}\right)
&=W\left(\tilde{\Phi_t(f_1)\left(\theta_{1,\frac{t}{\gd}}\right)},\tilde{ \Phi_t(f_1)\left(\theta_{3,\frac{t}{\gd}}\right)}\right) \nonumber\\
&\leq  \kappa_7 \gd^{ 2} \left\vert \theta_{1,\frac{t}{\gd}}- \theta_{3,\frac{t}{\gd}}\right\vert\, ,\ \text{ (by Lemma~\ref{lem:lipsch})} \nonumber\\
&=\kappa_7 \gd^{ 2} \left\vert \theta_{2,\frac{t}{\gd}}- \theta_{3,\frac{t}{\gd}}\right\vert\nonumber\\
&\leq C_\proj  \kappa_7 \gd \left\vert m_{2,\frac{t}{\gd}}- m_{3,\frac{t}{\gd}}\right\vert\, , \ \text{(using \eqref{eq:proj_lip})}\nonumber\\
&\leq C_\proj  \kappa_7 \gd^{1-\beta}\,  W\left(\mu_{2,\frac{t}{\gd}},\mu_{3,\frac{t}{\gd}}\right)\nonumber\\
& \leq  C_\proj \kappa_3 \kappa_7 \gd^{1-\beta} \,  W\left(\mu_{2,0},\mu_{3,0}\right)\, , \ \text{(by Lemma~\ref{lem:bound delta bar X delta bar Y})}\nonumber\\
& \leq C_\Pi C_\proj \kappa_3 \kappa_7 \gd^{1-\beta}\, W_\theta(f_2(\theta_{2,0}),f_1(\theta_{2,0}))\, ,
\end{align}
while on the other hand, Lemma~\ref{lem:bound delta bar X delta bar Y} implies that
\begin{align}
W\left(\tilde \mu_{3,\frac{t}{\gd}},\tilde \mu_{2,\frac{t}{\gd}}\right) &\leq  \max\left( e^{-(K- \delta\kappa_4)\frac{t}{\gd}},e^{-\frac{b}{c }t }(1+ b \kappa_6 \gd)\right)W\left(\tilde \mu_{3,0},\tilde \mu_{2,0}\right)\nonumber\\
& \qquad \qquad \qquad \qquad\qquad \qquad \qquad\qquad \qquad +\kappa_5\gd^{1-\beta} \, W(\mu_{3,0},\mu_{2,0})\nonumber\\
& \leq \left( e^{-\frac{b}{c }t }+ 2 C_\Pi \kappa_5 \gd^{1-\beta}\right)W_\theta (f_1(\theta_{2,0}),f_2(\theta_{2,0})),
\end{align}
where we have taken $\gd$ small enough.
So, for $\gd$ small enough,
\begin{equation}
W_\theta \left(\tilde \mu_{1,\frac{t}{\gd}},\tilde \mu_{2,\frac{t}{\gd}}\right)\leq \left(e^{-\frac{b}{c}t}+O\left(\gd^{1-\beta}\right)\right)\dist_\cF(f_1,f_2).
\end{equation}
Similarly, we have the decomposition, since $\theta_{1,\frac{t}{\gd}}=\theta_{2,\frac{t}{\gd}}= \theta$,
\begin{align}
\left\vert p^\perp_\theta\left(m_{1,\frac{t}{\gd}}-m_{2,\frac{t}{\gd}}\right)\right\vert &=  \left\vert p^\perp_{\theta_{1,\frac{t}{\gd}}}\left(m_{1,\frac{t}{\gd}}-\gamma_{\theta_{1,\frac{t}{\gd}}}\right)-p^\perp_{\theta_{2,\frac{t}{\gd}}}\left(m_{2,\frac{t}{\gd}}-\gamma_{\theta_{2,\frac{t}{\gd}}}\right)\right\vert \nonumber\\
&\leq  \left\vert p^\perp_{\theta_{1,\frac{t}{\gd}}}\left(m_{1,\frac{t}{\gd}}-\gamma_{\theta_{1,\frac{t}{\gd}}}\right)-p^\perp_{\theta_{3,\frac{t}{\gd}}}\left(m_{3,\frac{t}{\gd}}-\gamma_{\theta_{3,\frac{t}{\gd}}}\right)\right\vert\nonumber\\
&\quad  +\left\vert p^\perp_{\theta_{3,\frac{t}{\gd}}}\left(m_{3,\frac{t}{\gd}}-\gamma_{\theta_{3,\frac{t}{\gd}}}\right)-p^\perp_{\theta_{2,\frac{t}{\gd}}}\left(m_{2,\frac{t}{\gd}}-\gamma_{\theta_{2,\frac{t}{\gd}}}\right)\right\vert.
\end{align}
On one hand, using similar bounds as above, we get
\begin{align}
\bigg\vert p^\perp_{\theta_{1,\frac{t}{\gd}}}\left(m_{1,\frac{t}{\gd}}-\gamma_{\theta_{1,\frac{t}{\gd}}}\right)-&p^\perp_{\theta_{3,\frac{t}{\gd}}}\left(m_{3,\frac{t}{\gd}}-\gamma_{\theta_{3,\frac{t}{\gd}}}\right)\bigg\vert \nonumber\\
& \leq \kappa_8 \gd^\frac32 \left\vert \theta_{1,\frac{t}{\gd}}-\theta_{3,\frac{t}{\gd}}\right\vert \nonumber\\
& \leq C_\Pi C_\proj \kappa_3 \kappa_8 \gd^{\frac12-\beta} \, W_\theta(f_2(\theta_{2,0}),f_1(\theta_{2,0})).
\end{align}
On the other hand, either $\left\vert \theta_{3,\frac{t}{\gd}}-\theta_{2,\frac{t}{\gd}}\right\vert\geq \gd^{ - \frac{ 1}{ 2}}$, and in this case relying on Lemma~\ref{lem:close to q0} we have
\begin{align}
\bigg\vert p^\perp_{\theta_{3,\frac{t}{\gd}}}&\left(m_{3,\frac{t}{\gd}}-\gamma_{\theta_{3,\frac{t}{\gd}}}\right)- p^\perp_{\theta_{2,\frac{t}{\gd}}}\left(m_{2,\frac{t}{\gd}}-\gamma_{\theta_{2,\frac{t}{\gd}}}\right)\bigg\vert \nonumber\\
& \leq 2 \kappa_2 \gd \nonumber\\
& \leq 2\kappa_2 \gd^\frac32 \left\vert \theta_{3,\frac{t}{\gd}}-\theta_{2,\frac{t}{\gd}}\right\vert \nonumber\\
&\leq 2 C_\Pi C_\proj \kappa_2 \kappa_3 \gd^{\frac12 -\beta}  W_\theta(f_2(\theta_{2,0}),f_1(\theta_{2,0})),
\end{align}
or $\left\vert \theta_{3,\frac{t}{\gd}}-\theta_{2,\frac{t}{\gd}}\right\vert\leq \gd^{ - \frac{ 1}{ 2}}$, and
with similar arguments as used in the proofs of Lemma~\ref{lem:bound delta bar X delta bar Y} and Lemma~\ref{lem:encadre}, we have
\begin{align}
\bigg\vert p^\perp_{\theta_{3,\frac{t}{\gd}}}&\left(m_{3,\frac{t}{\gd}}-\gamma_{\theta_{3,\frac{t}{\gd}}}\right)- p^\perp_{\theta_{2,\frac{t}{\gd}}}\left(m_{2,\frac{t}{\gd}}-\gamma_{\theta_{2,\frac{t}{\gd}}}\right)\bigg\vert \nonumber\\
& =\left\vert p^\perp_{\theta_{2,0}+t}\left(m_{3,\frac{t}{\gd}}-m_{2,\frac{t}{\gd}}\right)\right\vert+O\left(\gd \left\vert \theta_{3,\frac{t}{\gd}}-\theta_{2,\frac{t}{\gd}}\right\vert \right)+O\left(\left\vert \theta_{3,\frac{t}{\gd}}-\theta_{2,\frac{t}{\gd}}\right\vert^2 \right)\nonumber \\
& = \left\vert p^\perp_{\theta_{2,0}+t}\left(m_{3,\frac{t}{\gd}}-m_{2,\frac{t}{\gd}}\right)\right\vert+O\left(\gd^\frac12 W(f_2(\theta_{2,0}),f_1(\theta_{2,0})) \right).
\end{align}
Relying on the decomposition \eqref{eq:decomp Delta m} and its following estimates (considering some coupling $ \left\lbrace \left( X_{ 3, 0}, Y_{ 3, 0}\right), \left( X_{ 2, 0}, Y_{ 2, 0}\right)\right\rbrace \sim \pi_{ 0}\in \mathcal{ C}(\mu_{ 3, 0}, \mu_{ 2, 0})$ of the initial condition such that $\Lambda(\pi_{ 0}) < W(\mu_{ 3,0},\mu_{ 2, 0}) + \varepsilon$),
\begin{align}
\bigg\vert p^\perp_{\theta_{2,0}+t}&\left(m_{3,\frac{t}{\gd}}-  m_{2,\frac{t}{\gd}}\right)\bigg\vert \nonumber\\
& \leq  e^{-\lambda t}\bigg\vert p^\perp_{\theta_{2,0}}(m_{3,0}-m_{2,0})\bigg\vert
+ C_\Pi(c_1+c_2+c_3)\kappa_3 t\gd^{1-\beta} \, W(f_1(\theta_{2,0}),f_2(\theta_{2,0})) \nonumber \\ 
& \qquad \qquad +C_\Pi c_4 \gd \int_0^{\frac{t}{\gd}}\left( \bbE\left[\left(\tilde X_{3,s}-\tilde X_{2,s}\right)^2\right]^\frac12 +b\bbE\left[\left(\tilde Y_{3,
s}-\tilde Y_{2,s}\right)^2\right]^\frac12\right)\dd s\nonumber\\
&\leq  \left(\gd^{-\beta}e^{-\gl t} + C\left(1+\gd^{1-\beta}\right)(1+\gep)\right) W_\theta(f_1(\theta_{2,0}),f_2(\theta_{2,0})) ,
\end{align}
for some positive constant $C$, where we have relied on Lemma~\ref{lem:bound delta bar X delta bar Y} to deal with the integral term.
Making $\gep$ going to $0$ and gathering these estimates we get for $\gd$ small enough:
\begin{equation}
\gd^\beta \left\vert p^\perp_\theta\left(m_{1,\frac{t}{\gd}}-m_{2,\frac{t}{\gd}}\right)\right\vert \leq \left(e^{-\gl t} +O\left(\gd^\frac12+\gd^\beta\right)\right)\dist_\cF(f_1,f_2),
\end{equation}
which concludes the proof.
\end{proof}

We now have all the ingredients needed to prove Theorem \ref{th:main}.

\begin{proof}[Proof of Theorem \ref{th:main}]
By Lemma \ref{lem:lipsch} we have $\Phi_T(\cF)\subset \cF$. Moreover, recalling \eqref{hyp T}, Lemma \ref{lem:X contracts} implies that $\Phi_T$ admits a unique fixed-point $f_0$ in $\cF$. The point is now to show that $\Phi_t(f_0)=f_0$ for all $t\geq 0$. It is clear that $\Phi_{kT}(f_0)=f_0$ for all $k\in \bbN$, by the semi-group property, and it remains then to prove that $\Phi_t(f_0)=f_0$ for $t\in(0,T)$. But, fixing $t\in (0,T)$, $\Phi_t(f_0)\in \cF$ by Lemma \ref{lem:lipsch} and we have $\Phi_{t} (f_0)=\Phi_t(\Phi_T (f_0))=\Phi_T(\Phi_t(f_0))$, so that $\Phi_t(f_0)=f_0$ by uniqueness of the fixed-point of $\Phi_T$ on $\cF$.

This proves that $\cC_\gd=\{f_0(\theta):\, \theta\in \bbS_\gd\}$ is a one-dimensional invariant manifold for \eqref{eq:PDER kinetic FhN}, and the estimates made in the proof of Lemma \ref{lem:close to q0}, showing that if $\mu_0=f_0(\theta_0)$ then $\theta_t=1+O(\gd)$, imply that $\cC_\gd$ defines in fact a periodic solution for \eqref{eq:PDER kinetic FhN}. Moreover from\eqref{def f close q0} and \eqref{def f close gamma} we deduce that for all $\theta\in \bbS_\gd$ we have $W(\tilde{f_0(\theta)},\tilde{q_0})\leq \kappa_1 \gd$, and $\left\vert \int_{\bbR^2}zf_0(\dd z)-\gamma_\theta\right\vert=O(\gd)$, which means that $W_2(f(\theta),q_{\gamma_\theta})=O(\gd)$.

\medskip

To prove the contraction property, first remark that, by equivalence of distances, if $\dist_{W_2}(\mu, \cC_\gs)\leq c\gd$ with $c$ small enough, then for $\theta_0=\proj(m_0)$ we have
$m_0\in E_{\theta_0}(\kappa_2\gd)$ and $W(\tilde \mu_0,q_0)\leq \kappa_1 \gd$, and for $f$ defined as
\begin{equation}
\tilde{f(\theta)}=\tilde \mu_0,\qquad \text{and} \qquad \int_{\bbR^2}z  f(\theta)(dz)=\gamma_\theta+N(\theta,\theta_0)\left(m_0-\gamma_{\theta_0} \right),
\end{equation}
we have $f\in \cF$. 
Now, on one hand, for $t\in[0,T]$, Lemma \ref{lem:bound delta bar X delta bar Y} implies that if $y^\gd_t$ is the periodic solution defining $\cC_\gd$, with $y^\gd_0=f_0(\theta_0)$, then $W(\mu_t,y^\gd_t)\leq \kappa_3 W(\mu_0,y^\gd_0)$, and thus for $\in[0,T]$:
\begin{equation}
\dist_{W_2}(\mu_t,\cC_\gd)\leq 3\kappa_3 \gd^{-\beta} W(\mu_0,f_0(\theta_0))\leq \frac{3\kappa_3}{C_\Pi}\gd^{-\beta} W_{\theta_0}(\mu_0,f_0(\theta_0)).
\end{equation}
On the other hand, since $\mu_t=\Phi_t(f)(\theta_t)$ and relying on the proof on Lemma \ref{lem:X contracts} and on a basic recursion, we get that for $t\geq T$:
\begin{equation}
\dist_{W_2}(\mu_t,\cC_\gd)\leq \frac{3\kappa_3}{C_\Pi}\gd^{-\beta}W_{\theta_t}(\Phi_t(f)(\theta_t),(f_0)(\theta_t))\leq \frac{3\kappa_3}{C_\Pi}\gd^{-\beta} e^{-\lambda(\theta) t}W_{\theta_0}(\mu_0,f_0(\theta_0)),
\end{equation}
where we have defined (using the bound $e^{-at}+xt\leq e^{-(a-xe^{at})t}$):
\begin{equation}
\lambda(\theta):=\min\left(\lambda,\frac{b}{c}\right)-\frac{\kappa_{10}e^{2T\min\left(\lambda,\frac{b}{c}\right)}}{T}(\gd^\beta+\gd^{1-\beta}).
\end{equation}
It remains thus to bound $W_{\theta_0}(\mu_0,f_0(\theta_0))$ by $\dist_{W_2}(\mu_0,\cC_\gd)$. To do this we show that for a constant $c$ we have $W_{\theta_0}(\mu_0,f(\theta_0)\leq c W(\mu_0,f_0(\theta))$ for all $\theta\in\bbS_\gd$. But we have on one hand, since $m_0\in E_{\theta_0}(\kappa_2\gd)$,
\begin{equation}
\left\vert p^\perp_{\theta_0}\left(m_0-\gamma_{\theta_0}\right)\right\vert\leq C_{\proj}\left\vert m_0-\gamma_{\theta_0}\right\vert\leq C_{\proj}\left\vert m_0-\gamma_{\theta}\right\vert+C_{\proj}\left\vert \gamma_{\theta}-\gamma_{\theta_0}\right\vert,
\end{equation}
and an application of Lemma \ref{lem:proj on gamma} leads to
\begin{equation}
\theta_0-\theta =\frac{1}{\gd}p_{\theta}\left(m_0-\gamma_{\theta}\right)+\frac{1}{\gd} O\left(\left\vert m_0-\gamma_{\theta}\right\vert^2\right).
\end{equation}
So, for a constant $C$ that does not depend on $\gd$,
$
\left\vert \gamma_{\theta}-\gamma_{\theta_0}\right\vert=\left\vert \gamma^1_{\gd\theta}-\gamma^1_{\gd\theta_0}\right\vert \leq C\left\vert m_0-\gamma_{\theta}\right\vert
$, and thus
\begin{equation}
\left\vert p^\perp_{\theta_0}\left(m_0-\gamma_{\theta_0}\right)\right\vert\leq C' \left\vert m_0-\gamma_{\theta}\right\vert,
\end{equation}
for some constant $C'$ that does not depend on $\gd$. On the other hand,
\begin{equation}
W\left(\tilde \mu_0,\tilde{f_0(\theta_0)}\right)\leq W\left(\tilde \mu_0,\tilde{f_0(\theta)}\right)+W\left(\tilde{f_0(\theta)},\tilde{f_0(\theta_0)}\right),
\end{equation}
and $W\left(\tilde{f_0(\theta)},\tilde{f_0(\theta_0)}\right)\leq \kappa_7 \gd^2\vert \theta-\theta_0\vert\leq C''\left\vert m_0-\gamma_{\theta}\right\vert $. Gathering all these estimates, we obtain $W_{\theta_0}(\mu_0,f(\theta_0)\leq c W(\mu_0,f_0(\theta))$, with $c=\max(C',1+C'')$.

Finally, for all $t\geq 0$ and $\theta\in \bbS_\gd$ we get:
\begin{equation}
\dist_{W_2}(\mu_t,\cC_\gd)\leq C(\gd) e^{-\lambda(\gd)t}W_2(\mu_0,f_0(\theta)),
\end{equation}
with $C(\gd)= \frac{3c}{C_\Pi^2\min(1,b^{-1})}\gd^{-\beta}\max\left(\kappa_3 e^{\lambda(\gd)T}, 1\right)$, which concludes the proof.
\end{proof}

\appendix

\section{Technical Lemma}

\begin{lemma}\label{lem:gronwall_sqrt}
Let $v$ be a continuously differentiable function on $[0, +\infty)$ such that $v(t)\geq0$ for all $t\geq0$. Suppose that there exists $ \alpha, \beta,  k>0$ and $\zeta\geq 0$ such that
\begin{equation}
v^{ \prime}(t)\leq - \alpha v(t) +\left(e^{-k t} \zeta +\beta\right) \sqrt{ v(t)}.
\end{equation}
Then, for all $t\geq 0$,
\begin{equation}\label{eq:gronwall_sqrt_1}
v(t) \leq \max \left(v(0),\ \frac{ (\zeta +\beta)^{ 2}}{ \alpha^{ 2}}\right), 
\end{equation}
and if $2k>\ga$ or $\zeta=0$, then
\begin{equation}\label{eq:gronwall_sqrt_2}
v(t) \leq \left(e^{ - \alpha t/2} \left(\sqrt{ v(0)}+\frac{\zeta}{2k-\ga }\right) +\frac{\beta}{\ga } \right)^{ 2}. 
\end{equation}
\end{lemma}

\begin{proof}[Proof of Lemma~\ref{lem:gronwall_sqrt}]
The first is a consequence of the fact that $v(t)$ is always non-increasing unless $ \sqrt{ v(t)}\leq \frac{ e^{-kt}\zeta +\beta}{ \alpha}\leq \frac{\zeta+\beta}{\ga}$. We now prove the second inequality: let $t\geq0$ such that $v(t)>0$ and consider the maximal interval $I:=(t^{ -}, t^{ +})$ (which is non empty by continuity of $v$) containing $t$ such that $v(u)>0$ on $I$. Consider the function $f(u):= \left( \alpha \sqrt{ v(u)} +\frac{\ga}{2k-\ga}e^{-ku}\zeta - \beta\right)e^{ \alpha u/2}$. For all $u\in I$ we have
\begin{equation}
f^{ \prime}(u) = \frac{ \alpha e^{ \alpha u/2}}{ 2 \sqrt{ v(u)}} \left( \alpha v(u) -\left(e^{-ku}\zeta+ \beta\right) \sqrt{ v(u)} + v^{ \prime}(u)\right)\leq 0,
\end{equation}
so that $f$ is nonincreasing on $I$. Consider now the solution $w$ of the equation $w^{ \prime}(t)= - \alpha w(t) + \left(e^{-kt}\zeta+\beta\right) \sqrt{ w(t)}$ such that $w(t^{ -})=v(t^{ -})$. Then, relying on the same calculations, $g:= u \mapsto \left( \alpha \sqrt{ w(u)}+\frac{\ga}{2k-\ga}e^{-ku}\zeta - \beta\right)e^{ \alpha u/2}$ is constant on $I$, and equal to $g(t^{ -})$. This means that $f(u) \leq f(t^{ -})= g(t^{ -})= g(u)$ and, by definition of $f$ and $g$, this implies that
\begin{multline}
v(u) \leq w(u)=   \left(\frac{e^{ - \alpha u/2}}{\ga} g(t^{ -}) -\frac{e^{-ku}}{2k-\ga}\zeta+ \frac{\beta}{\ga} \right)^{ 2}\\
= \left(e^{-\ga(u-t^-)/2}\sqrt{v(t^-)}+\frac{e^{-\ga(u-t^-)/2}e^{-kt^-}-e^{-ku}}{2k-\ga}\zeta+ \left(1-e^{ \alpha (u-t^-)/2}\right)\frac{\beta}{\ga} \right)^{ 2}.
\end{multline}
We have now two possibilities: either $t^{ -}=0$ and then
\begin{equation}
v(u) \leq \left( e^{ - \alpha u/2}\sqrt{ v(0)} +\frac{e^{-\ga u/2}-e^{-ku}}{2k-\ga}\zeta+\left(1-e^{ -\alpha u/2}\right)\frac{\beta}{\ga}\right)^{ 2},
\end{equation}
or $t^{ -}>0$, and by continuity of $v$ we have $v(t^{ -})=0$. In this case, recalling that $2k>\ga$ or $\zeta=0$, we obtain
\begin{equation}
v(u) \leq   \left(\frac{e^{-\ga u/2}}{2k-\ga}\zeta+\frac{\beta}{\ga}\right)^{ 2}.
\end{equation}
This proves Lemma~\ref{lem:gronwall_sqrt}.
\end{proof}

\section*{Acknowledgements}

C. Poquet benefited from the support of the ANR-17-CE40-0030 (Entropy, Flows, Inequalities).

\end{document}